\documentclass[10pt,a4paper]{amsart}

\usepackage{fullpage}
\usepackage{amsmath}
\usepackage{amsfonts}
\usepackage{amssymb}
\usepackage{tikz-cd}
\usepackage{url}
\usepackage[english]{babel}
\usepackage[latin1]{inputenc}
\usepackage{enumitem}

\newtheorem{theorem}{Theorem}[section]
\newtheorem{definition}[theorem]{Definition}
\newtheorem{remark}[theorem]{Remark}

\newtheorem{proposition}[theorem]{Proposition}
\newtheorem{corollary}[theorem]{Corollary}

\newtheorem{lemma}[theorem]{Lemma}

\newcommand{\A}{\mathbb A}
\newcommand{\B}{\mathbb B}
\newcommand{\C}{\mathbb C}
\newcommand{\F}{\mathbb F}
\newcommand{\G}{\mathbb G}
\newcommand{\I}{\mathbb I}

\newcommand{\N}{\mathbb N}
\newcommand{\Q}{\mathbb Q}

\newcommand{\T}{\mathbb T}
\newcommand{\Z}{\mathbb Z}

\newcommand{\bP}{{\bf P}}
\newcommand{\bt}{{\bf t}}

\newcommand{\cB}{\mathcal{B}}
\newcommand{\calH}{\mathcal{H}}

\newcommand{\calC}{\mathcal{C}}

\newcommand{\cG}{\mathcal{G}}

\newcommand{\calK}{\mathcal{K}}

\newcommand{\cO}{\mathcal{O}}
\newcommand{\cR}{\mathcal{R}}

\newcommand{\cW}{\mathcal{W}}

\newcommand{\ad}{\mathrm{ad}}
\newcommand{\Ad}{\mathrm{Ad}}

\newcommand{\an}{{\mathrm an}}
\newcommand{\Aut}{\mathrm{Aut}}
\newcommand{\diag}{\mathrm{diag}}
\newcommand{\End}{\mathrm{End}}

\newcommand{\Frob}{\mathrm{Frob}}
\newcommand{\Gal}{\mathrm{Gal}}
\newcommand{\GL}{\mathrm{GL}}
\newcommand{\GSp}{\mathrm{GSp}}

\newcommand{\Id}{\mathrm{Id}}
\newcommand{\im}{\mathrm{Im}}
\newcommand{\Ind}{\mathrm{Ind}}

\newcommand{\loc}{\mathrm{loc}}
\newcommand{\Log}{\mathrm{Log}}
\newcommand{\Mat}{\mathrm{M}}

\newcommand{\SL}{\mathrm{SL}}
\newcommand{\Spec}{\mathrm{Spec}}
\newcommand{\Spf}{\mathrm{Spf}}
\newcommand{\Spm}{\mathrm{Spm}}
\newcommand{\Tr}{{\mathrm{Tr}}}

\newcommand{\into}{\hookrightarrow}
\newcommand{\ccirc}{\kern0.5ex\vcenter{\hbox{$\scriptstyle\circ$}}\kern0.5ex}

\newcommand{\fA}{{\mathfrak{A}}}
\newcommand{\fa}{{\mathfrak{a}}}

\newcommand{\fc}{{\mathfrak{c}}}

\newcommand{\ff}{{\mathfrak{f}}}
\newcommand{\fg}{{\mathfrak{g}}}
\newcommand{\fH}{{\mathfrak{H}}}
\newcommand{\frakI}{{\mathfrak{I}}}

\newcommand{\fl}{{\mathfrak{l}}}
\newcommand{\m}{\mathfrak{m}}
\newcommand{\fm}{{\mathfrak{m}}}

\newcommand{\fp}{{\mathfrak{p}}}
\newcommand{\fP}{{\mathfrak{P}}}
\newcommand{\fq}{{\mathfrak{q}}}
\newcommand{\fQ}{{\mathfrak{Q}}}

\newcommand{\fs}{{\mathfrak{s}}}
\newcommand{\ft}{{\mathfrak{t}}}
\newcommand{\fu}{{\mathfrak{u}}}
\newcommand{\fU}{{\mathfrak{U}}}

\newcommand{\fsl}{{\mathfrak{sl}}}

\title{Big image of Galois representations associated with finite slope $p$-adic families of modular forms}
\author{Andrea Conti$^\ast$, Adrian Iovita, Jacques Tilouine$^\ast$}
\thanks{$^\ast$Supported by the Programs ArShiFo ANR-10-BLAN-0114 and PerColaTor ANR-14-CE25-0002-01}

\begin{document}
\maketitle

\section{Introduction}

Let $f$ be a non-CM cuspidal eigenform and let $\ell$ be a prime integer. By the work of Ribet (\cite{ribet1}, \cite{ribet3}) and Momose (\cite{momose}), it is known that the $\ell$-adic Galois representation $\rho_{f,\ell}$ associated with $f$ has large image for every $\ell$ and that for almost every $\ell$ it satisfies 
\medskip
\begin{center}
(cong${}_\ell$) $\im\,\rho_{f,\ell}$ contains the conjugate of a principal congruence subgroup $\Gamma(\ell^m)$ of $\SL_2(\Z_\ell)$.
\end{center}
\medskip
\noindent For instance if $\im\,\rho_{f,\ell}$ contains an element with eigenvalues in $\Z_\ell^\times$ distinct modulo $\ell$ then (cong${}_\ell$) holds.

\noindent In \cite{hida}, Hida proved an analogous statement for $p$-adic families of non-CM ordinary cuspidal eigenforms, where $p$ is any odd prime integer. We fix once and for all an embedding $\overline{\Q}\into\overline{\Q}_p$, identifying $\Gal(\overline{\Q}_p/\Q_p)$ with a decomposition subgroup $G_p$ of $\Gal(\overline{\Q}/\Q)$. We also choose a topological generator $u$ of $\Z_p^\times$. 
Let $\Lambda=\Z_p[[T]]$ be the Iwasawa algebra and let $\m=(p,T)$ be its maximal ideal. A special case of Hida's first main theorem  (\cite[Th.I]{hida}) is the following.
 
 \begin{theorem} Let ${\bf f}$ be a non-CM Hida family of ordinary cuspidal eigenforms defined over a finite extension $\I$ of $\Lambda$ and let $\rho_{\bf f}\colon \Gal(\overline{\Q}/\Q)\to\GL_2(\I)$ be the associated Galois representation. Assume that $\rho_{\bf f}$ is residually irreducible and that there exists an element $d$ in its image with eigenvalues $\alpha,\beta\in\Z_p^\times$ such that $\alpha^2\not\equiv\beta^2\pmod{p}$. Then there exists a nonzero ideal ${\mathfrak l}\subset \Lambda$ and an element $g\in\GL_2(\I)$ such that
$$ g\Gamma({\mathfrak l})g^{-1}\subset \im\,\rho_{\bf f}, $$
where $\Gamma({\mathfrak l})$ denotes the principal congruence subgroup of $\SL_2(\Lambda)$ of level ${\mathfrak l}$.
 \end{theorem}
 
Under mild technical assumptions it is also shown in \cite[Th. II]{hida} that if the image of the residual representation of $\rho_{\bf f}$ contains a conjugate of $\SL_2(\F_p)$ then ${\mathfrak l}$ is trivial or $\m$-primary, and if the residual representation is dihedral ``of CM type'' the height one prime factors $P$ of ${\mathfrak l}$ are exactly those of the g.c.d. of the adjoint $p$-adic $L$ function of ${\bf f}$ and the anticyclotomic specializations of Katz' $p$-adic $L$ functions associated with certain Hecke characters of an imaginary quadratic field. This set of primes is precisely the set of congruence primes between the given non-CM family and the CM families.
 
In her PhD dissertation (see \cite{lang}), J. Lang improved on Hida's Theorem I. Let $\T$ be Hida's big ordinary cuspidal Hecke algebra; it is finite and flat over $\Lambda$. Let $\Spec\,\I$ be an irreducible component of $\T$. It corresponds to a surjective $\Lambda$-algebra homomorphism $\theta\colon \T\to \I$ (a $\Lambda$-adic Hecke eigensystem). We also call $\theta$ a Hida family. Assume that it is not residually Eisenstein. It gives rise to a residually irreducible continuous Galois representation $\rho_\theta\colon G_\Q\to\GL_2(\I)$ that is $p$-ordinary.
We suppose for simplicity that $\I$ is normal. Consider the $\Lambda$-algebra automorphisms $\sigma$ of $\I$ for which there exists a finite order character $\eta_\sigma\colon G_\Q\to\I^\times$ 
such that for every prime $\ell$ not dividing the level, $\sigma\ccirc\theta(T_\ell)=\eta_\sigma(\ell) \theta(T_\ell)$
(see \cite{ribet3} and \cite{lang}). These automorphisms form a finite abelian $2$-group $\Gamma$. 
Let $\I_0$ be the subring of $\I$ fixed by $\Gamma$. Let $H_0=\bigcap_{\sigma\in\Gamma}\ker\,\eta_\sigma$; 
it is a normal open subgroup of  $G_\Q$. One may assume, up to conjugation by an element of $\GL_2(\I)$, 
that $\rho_\theta\vert_{H_0}$ takes values in $\GL_2(\I_0)$.

 \begin{theorem} \cite[Th. 2.4]{lang}
Let $\theta\colon\T\to\I$ be a non-CM Hida family such that $\overline{\rho}_\theta$ is absolutely irreducible. Assume that $\overline{\rho}_\theta\vert_{H_0}$ is an extension of two distinct characters. Then there exists a nonzero ideal ${\mathfrak l}\subset \I_0$ and an element $g\in\GL_2(\I)$ such that
$$g\Gamma({\mathfrak l})g^{-1}\subset \im\,\rho_{\theta},$$
where $\Gamma({\mathfrak l})$ denotes the principal congruence subgroup of $\SL_2(\I_0)$ of level ${\mathfrak l}$.
 \end{theorem}
 
For all of these results it is important to assume the ordinarity of the family, as it implies 
the ordinarity of the Galois representation and in particular that some element of the image of inertia at $p$ is conjugate to the matrix 
$$C_T=\left(\begin{array}{cc} u^{-1}(1+T)&\ast\\0&1\end{array}\right).$$
Conjugation by the element above defines a $\Lambda$-module structure on the Lie algebra 
of a pro-$p$ subgroup of $\im\,\rho_{\bf f}$ and this is used to produce the desired ideal ${\mathfrak l}$.
Hida and Lang use Pink's theory of Lie algebras of pro-$p$ subgroups of $\SL_2(\I)$.

In this paper we propose a generalization of Hida's work to the finite slope case. 
We establish analogues of Hida's Theorems I and II. These are Theorems \ref{betalevel}, \ref{comparison} 
and \ref{exponents} in the text.
Moreover, we put ourselves in the more general setting considered in Lang's work.
In the positive slope case the existence of a normalizing matrix analogous to $C_T$ above is obtained by 
applying relative Sen theory (\cite{sen1} and \cite{sen2}) to the expense of extending scalars to 
the completion $\C_p$ of an algebraic closure of $\Q_p$.

More precisely, for every $h\in (0, \infty)$, we define an Iwasawa algebra $\Lambda_h=\cO_h[[t]]$ 
(where $t=p^{-s_h}T$ for some $s_h\in\Q\cap ]0,\frac{1}{p-1}[$ and $\cO_h$ 
is a finite extension of $\Z_p$ containing $p^{s_h}$ such that its fraction field is Galois over $\Q_p$) and a finite torsion free 
$\Lambda_h$-algebra $\T_h$ (see Section \ref{subsectnest}), called an adapted
slope $\leq h$ Hecke algebra.
Let $\theta\colon\T_h\to\I^\circ$ be an irreducible component; it is finite torsion-free over $\Lambda_h$.
The notation $\I^\circ$ is borrowed from the theory of Tate algebras, but $\I^\circ$ is not a Tate or an affinoid algebra. We write $\I=\I^\circ[p^{-1}]$.
We assume again for simplicity that $\I^\circ$ is normal.
The finite slope family $\theta$ gives rise to a continuous Galois representation
$\rho_\theta\colon G_\Q\to\GL_2(\I^\circ)$. We assume that the residual representation $\overline{\rho_\theta}$ 
is absolutely irreducible.
We introduce the finite abelian $2$-group $\Gamma$ as above, together with its fixed ring $\I_0$ 
and the open normal subgroup $H_0\subset G_\Q$.
In Section \ref{liealg} we define a ring $\B_{r}$ (with an inclusion $\I_0\into\B_r$) and a Lie algebra $\fH_{r}\subset\fsl_2(\B_{r})$ attached to the image of $\rho_\theta$.
In the positive slope case CM families do not exist (see Section \ref{cmforms}) hence no ``non-CM'' assumption is needed in the following. As before we can assume, after conjugation by an element of $\GL_2(\I^\circ)$, that $\rho_\theta(H_0)\subset\GL_2(\I_0^\circ)$. Let $P_1\subset\Lambda_h$ be the prime $(u^{-1}(1+T)-1)$.

 \begin{theorem} \textnormal{(Theorem \ref{betalevel})}
Let $\theta\colon\T_h\to\I^\circ$ be a positive slope family such that $\overline{\rho}_\theta\vert_{H_0}$ is absolutely irreducible. Assume that there exists $d\in \rho_\theta(H_0)$ with eigenvalues $\alpha,\beta\in\Z_p^\times$ such that $\alpha^2\not\equiv\beta^2\pmod{p}$. Then there exists a nonzero ideal ${\mathfrak l}\subset \I_0[P_1^{-1}]$ such that
$$ \fl\cdot\fsl_2(\B_{r})\subset\fH_r. $$
 \end{theorem}
\noindent The largest such ideal $\fl$ is called the Galois level of $\theta$.

We also introduce the notion of fortuitous CM congruence ideal for $\theta$ (see Section \ref{congrid}). It is the ideal $\fc\subset \I$
given by the product of the primary ideals modulo which a congruence between $\theta$ and a slope $\leq h$ CM form occurs. 
Following the proof of Hida's Theorem II we are also able to show (Theorem \ref{comparison}) that the set of primes of $\I_0=\I_0^\circ[p^{-1}]$
containing $\fl$ coincides
with the set of primes containing $\fc\cap\I_0$, except possibly for the primes of $\I_0$ above 
$P_1$ (the weight $1$ primes). 

Several generalizations of the present work are currently being studied by one of the authors\footnote{A. Conti}. They include a generalization of \cite{hidatil}, where the authors treated the ordinary case for $\GSp_4$ with a residual representation induced from the one associated with a Hilbert modular form, to the finite slope case and to bigger groups and more types of residual representations.

\bigskip

\textbf{Acknowledgments.} This paper owes much to Hida's recent paper \cite{hida}. We also thank Jaclyn Lang for making her dissertation \cite{lang} available to us and for some very useful remarks pertaining to Section \ref{imgfinslope}. 
We thank the referee of this article for the careful reading of the manuscript and for useful suggestions which hopefully led to improvements.

\bigskip

\section{The eigencurve}


\subsection{The weight space}

Let us recall that we have fixed $p>2$ be a prime integer.
We call \textit{weight space} the rigid analytic space over $\Q_p$, $\cW$, canonically associated with the formal scheme over $\Z_p$, ${\rm Spf}(\Z_p[[\Z_p^\times]])$. The $\C_p$-points of $\cW$ parametrize continuous homomorphisms $\Z_p^\times\to\C_p^\times$.

Let $X$ be a rigid analytic space defined over some finite extension $L/\Q_p$. We say that a subset $S$ of $X(\C_p)$ is Zariski-dense if the only closed analytic subvariety $Y$ of $X$ satisfying $S\subset Y(\C_p)$ is $X$ itself.

%
For every $r>0$, we denote by $\cB(0,r)$, respectively $\cB(0,r^-)$, the closed, respectively open, disc in $\C_p$ of centre $0$ and radius $r$.
The space $\cW$ is isomorphic to a disjoint union of
$p-1$ copies of the open unit disc $\cB(0,1^-)$ centered in $0$ and indexed by the group $\Z/(p-1)\Z=\widehat{\mu}_{p-1}$. If $u$ denotes a topological generator of $1+p\Z_p$, then an isomorphism is given by
$$ \Z/(p-1)\Z\times \cB(0,1^-)\to\cW,\quad (i,v)\mapsto\chi_{i,v}, $$
where $\chi_{i,v}((\zeta,u^x))=\zeta^i (1+v)^x$. Here we wrote an element of $\Z_p^\times$ uniquely as a pair $(\zeta, u^x)$ with $\zeta\in \mu_{p-1}$ and $x\in \Z_p$. We make once and for all the choice $u=1+p$.


We say that a point $\chi\in\cW(\C_p)$ is classical if there exist $k\in\N$ and a finite order character $\psi\colon\Z_p^\times\to\C_p^\times$ such that $\chi$ is the character $z\mapsto z^k\psi(z)$. The set of classical points is Zariski-dense in $\cW(\C_p)$.

If $\Spm\, R\subset\cW$ is an affinoid open subset, we denote by $\kappa=\kappa_R\colon\Z_p^\times\to R^\times$ its tautological character 
given by $\kappa(t)(\chi)=\chi(t)$ for every $\chi\in\Spm\, R$.
Recall (\cite[Prop. 8.3]{buzzard}) that $\kappa_R$ is $r$-analytic for every sufficiently small radius $r>0$ (by which we mean that it extends to a rigid analytic function on $\Z_p^\times \cB(1,r)$).


\subsection{Adapted pairs and the eigencurve}\label{sectadapt}

Let $N$ be a positive integer prime to $p$. We recall the definition of the spectral curve $Z^N$ and of the cuspidal eigencurve $C^N$ of 
tame level $\Gamma_1(N)$. These objects were constructed in \cite{colmaz} for $p>2$ and $N=1$ and in \cite{buzzard} in general. We follow the presentation of \cite[Part II]{buzzard}.
Let $\Spm\,R \subset\cW$ be an affinoid domain and let $r=p^{-s}$ for $s\in\Q$ be a radius smaller than the radius of analyticity of $\kappa_R$.
We denote by $M_{R,r}$ the $R$-module of $r$-overconvergent modular forms of weight $\kappa_R$.
It is endowed it with a continuous action of the Hecke operators $T_\ell$, $\ell\nmid Np$, and $U_p$.
The action of $U_p$ on $M_{R,r}$ is completely continuous, so we can consider its associated Fredholm series 
$F_{R,r}(T)=\det(1-U_pT\vert M_{R,r})\in R\{\{T\}\}$.
These series are compatible when $R$ and $r$ vary, in the sense that there exists $F\in\Lambda\{\{T\}\}$ that restricts to $F_{R,r}(T)$ for every $R$ and $r$.


The series $F_{R,r}(T)$ converges everywhere on the $R$-affine line $\Spm\,R\times \A^{1,\an}$, so it defines
a rigid curve $Z^{N}_{R,r}=\{F_{R,r}(T)=0\}$ in $\Spm\,R\times\A^{1,\an}$. 
When $R$ and $r$ vary, these curves glue into a rigid space $Z^N$ endowed with 
a quasi-finite and flat morphism $w_Z\colon Z^N\to\cW$. The curve $Z^N$ is called the spectral curve
associated with the $U_p$-operator. 
For every $h\geq 0$, let us consider
$$ Z_R^{N,\leq h}=Z^N_R\cap\left(\Spm\,R\times B(0,p^h)\right). $$
By \cite[Lemma 4.1]{buzzard} $Z_R^{N,\leq h}$ is quasi-finite and flat over $\Spm\,R$.

We now recall how to construct an admissible covering of $Z^N$.

\begin{definition}
We denote by $\calC$ the set of affinoid domains $Y\subset Z$ such that:
\begin{itemize}
\item there exists an affinoid domain $\Spm\,R\subset\cW$ such that $Y$ is a union of connected components of $w_Z^{-1}(\Spm\,R)$;
\item the map $w_Z\vert_Y\colon Y\to\Spm\,R$ is finite.
\end{itemize}
\end{definition}

\begin{proposition} \cite[Th. 4.6]{buzzard}
The covering $\calC$ is admissible.
\end{proposition}

Note in particular that an element $Y\in\calC$ must be contained in $Z_R^{N,\leq h}$ for some $h$.

For every $R$ and $r$ as above and every $Y\in\calC$ such that $w_Z(Y)=\Spm\,R$, we can associate with $Y$ a direct factor $M_Y$ of $M_{R,r}$ by the construction in \cite[Sec. I.5]{buzzard}.
The abstract Hecke algebra $\calH=\Z[T_\ell]_{\ell\nmid Np}$ acts on $M_{R,r}$ and $M_Y$ is stable with respect to this action.
Let $\T_Y$ be the $R$-algebra generated by the image of $\calH$ in $\End_R(M_Y)$ and let $C_Y^N=\Spm\,\T_Y$.
Note that it is reduced as all Hecke operators are self-adjoint for a certain pairing and mutually commute.

For every $Y$ the finite covering $C_Y^N\to\Spm\,R$ factors through $Y\to\Spm\,R$.
The eigencurve $C^N$ is defined by gluing the affinoids $C_Y^N$ into a rigid curve, endowed with a finite morphism $C^N\to Z^N$.
The curve $C^N$ is reduced and flat over $\cW$ since it is so locally.


We borrow the following terminology from Bella\"\i che.

\begin{definition}\label{adapt} \cite[Def. II.1.8]{bell} 
Let $\Spm\,R\subset\cW$ be an affinoid open subset and $h>0$ be a rational number.
The couple $(R,h)$ is called adapted if $Z_R^{N,\leq h}$ is an element of $\calC$.
\end{definition}

\noindent By \cite[Cor. II.1.13]{bell} the sets of the form $Z_R^{N,\leq h}$ are sufficient to admissibly cover the spectral curve.

Now we fix a finite slope $h$. We want to work with families of slope $\leq h$ which are finite over a wide open subset of the weight space.
In order to do this it will be useful to know which pairs $(R,h)$ in a connected component of $\cW$ are adapted.
If $\Spm\,R^\prime\subset\Spm\,R$ are affinoid subdomains of $\cW$ and $(R,h)$ is adapted then $(R^\prime,h)$ is also adapted 
by \cite[Prop. II.1.10]{bell}. By \cite[Lemma 4.3]{buzzard}, the affinoid $\Spm\,R$ is adapted to $h$ if and only if the weight map $Z_R^{N,\leq h}\to\Spm\, R$ has fibres of constant degree.

\begin{remark}
Given a slope $h$ and a classical weight $k$, it would be interesting to have a lower bound for the radius of a disc of centre $k$ adapted to $h$. A result of Wan \textnormal{(\cite[Th. 2.5]{wan})} asserts that for a certain radius $r_h$ depending only on $h,N$ and $p$, the degree of the fibres of $Z_{\cB(k,r_h)}^{N,\leq h}\to\Spm\, \cB(k,r_h)$ at classical weights is constant. Unfortunately we do not know whether the degree is constant at all weights of $\cB(k,r_h)$, so this is not sufficient to answer our question. Estimates for the radii of adapted discs exist in the case of eigenvarieties for groups different than $\GL_2$; see for example the results of Chenevier on definite unitary groups \textnormal{(\cite[Sec. 5]{chenevier})}.
\end{remark}

\subsection{Pseudo-characters and Galois representations}

Let $K$ be a finite extension of $\Q_p$ with valuation ring $\cO_K$.
Let $X$ be a rigid 
analytic variety defined over $K$.
We denote by $\cO(X)$ the ring of global analytic functions on $X$ equipped with the coarsest locally convex topology making the restriction map $\cO(X)\to\cO(U)$ continuous for every affinoid $U\subset X$. It is a Fr\'echet space isomorphic to the inverse limit over all affinoid domains $U$ of the $K$-Banach spaces $\cO(U)$. We denote by $\cO(X)^\circ$ the $\cO_K$-algebra of functions bounded by $1$ on $X$, equipped with the topology induced by that on $\cO(X)$. The question of the compactness of this ring is related to the following property of $X$.

\begin{definition}\label{defnested}
\cite[Def. 7.2.10]{bellchen}
We say that a rigid 
analytic variety $X$ defined over $K$ is nested if there is an admissible covering $X=\bigcup X_i$ by open affinoids $X_i$ 
defined over $K$ such that the maps $\cO(X_{i+1})\to\cO(X_i)$ induced by the inclusions are compact.
\end{definition}

We equip the ring $\cO(X)^\circ$ with the topology induced by that on $\cO(X)=\varprojlim_i \cO(X_i)$.

\begin{lemma}\label{nestcpt}
\cite[Lemma 7.2.11(ii)]{bellchen}
If $X$ is reduced and nested, then $\cO(X)^\circ$ is a compact (hence profinite) $\cO_K$-algebra. 
\end{lemma}

%

We will be able to apply Lemma \ref{nestcpt} to the eigenvariety thanks to the following.

\begin{proposition}
\cite[Cor. 7.2.12]{bellchen}
The eigenvariety $C^N$ is nested for $K=\Q_p$.
\end{proposition}

Given a reduced nested subvariety $X$ of $C^N$ defined over a finite extension $K$ of $\Q_p$ there is a pseudo-character on $X$ obtained by interpolating the classical ones.

\begin{proposition}\label{pseudochar}
\cite[Th. IV.4.1]{bell}
There exists a unique pseudo-character
$$ \tau\colon G_{\Q,Np}\to \cO(X)^\circ $$
\noindent of dimension $2$ such that for every $\ell$ prime to $Np$, $\tau(\Frob_\ell)=\psi_X(T_\ell)$, where $\psi_X$ is the composition of $\psi\colon\calH\to\cO(C^N)^\circ$ with the restriction map $\cO(C^N)^\circ\to\cO(X)^\circ$.
\end{proposition}

\begin{remark} One can take as an example of $X$ a union of irreducible components of $C^N$ in which case $K=\Q_p$.
Later we will consider other examples where $K\neq\Q_p$.
\end{remark}

\bigskip

\section{The fortuitous congruence ideal}\label{sectcong}


In this section we will define families with slope bounded by a finite constant and coefficients in a suitable profinite ring.
We will show that any such family admits at most a finite number of classical specializations which are CM modular forms.
Later we will define what it means for a point (not necessarily classical) to be CM and we will associate with a family a congruence ideal describing its CM points.
Contrary to the ordinary case, the non-ordinary CM points do not come in families so the points detected by the congruence ideal do not correspond to a crossing between a CM and a non-CM family.
For this reason we call our ideal the ``fortuitous congruence ideal''.

\subsection{The adapted slope $\leq h$ Hecke algebra}\label{subsectnest}

Throughout this section we fix a slope $h>0$ .
Let $C^{N,\leq h}$ be the subvariety of $C^N$ whose points have slope $\leq h$. Unlike the ordinary case treated in \cite{hida} the weight map $w^{\leq h}\colon C^{N,\leq h}\to\cW$ is not finite which means that a family of slope $\leq h$ is not in general defined by a finite map over the entire weight space. The best we can do in the finite slope situation is to place ourselves over the largest possible wide open subdomain $U$ of $\cW$ such that the restricted weight map $w^{\leq h}\vert_U\colon C^{N,\leq h}\times_{\cW}U\to U$ is finite. This is a domain ``adapted to $h$'' in the sense of Definition \ref{adapt} where only affinoid domains were considered.
The finiteness property will be necessary in order to apply going-up and going-down theorems.


Let us fix a rational number $s_h$ such that for $r_h=p^{-s_h}$ the closed disc $B(0,r_h)$ is adapted for $h$.
We assume that $s_h>\frac{1}{p-1}$ (this will be needed later to assure the convergence of the exponential map). Let $\eta_h\in\overline{\Q}_p$ be an element 
of $p$-adic valuation $s_h$. Let $K_h$ be the Galois closure (in $\C_p$) of $\Q_p(\eta_h)$ and let $\cO_h$ be its valuation ring. Recall that $T$ is the variable on the open disc of radius $1$. Let $t=\eta_h^{-1}T$ and $\Lambda_h=\cO_h[[t]]$.
This is the ring of analytic functions, with $\cO_h$-coefficients and bounded by one, on the wide open disc $\cB_h$ of radius $p^{-s_h}$.
There is a natural map $\Lambda\to\Lambda_h$ corresponding to the restriction of analytic functions on the open disc of radius $1$, with $\Z_p$ coefficients and bounded by $1$, to the open disc of radius $r_h$. The image of this map is the ring $\Z_p[[\eta t]]\subset\cO_h$.

For $i\geq 1$, let $s_i=s_h+1/i$ and $\cB_i=\cB(0,p^{-s_i})$. The open disc $\cB_h$ is the increasing union of the affinoid discs $\cB_i$. 
For each $i$ a model for $\cB_i$ over $K_h$ is given by Berthelot's construction of $\cB_h$ as 
the rigid space associated with the $\cO_h$-formal scheme $\Spf\,\Lambda_h$. We recall it briefly following \cite[Sec. 7]{dejong}. 
Let 
$$ A_{r_i}^\circ=\cO_h\langle t,X_i\rangle/(pX_i-t^i). $$

We have $\cB_i=\Spm\,A_{r_i}^\circ[p^{-1}]$ as rigid space over $K_h$.
For every $i$ we have a morphism $A_{r_{i+1}}^\circ\to A_{r_i}^\circ$ given by
$$ X_{i+1}\mapsto X_i t $$
$$ t\mapsto t $$

We have induced morphisms $A_{r_{i+1}}^\circ[p^{-1}]\to A_{r_i}^\circ[p^{-1}]$, hence open immersions $\cB_i\to\cB_{i+1}$ defined over $K_h$. The wide open disc $\cB_h$ is defined as the inductive limit of the affinoids $\cB_i$ with these transition maps.
We have $\Lambda_h=\varprojlim_i A_{r_i}^\circ$.



Since the $s_i$ are strictly bigger than $s_h$ for each $i$, $\cB(0,p^{-s_i})=\Spm\, A_{r_i}^\circ[p^{-1}]$ is adapted to $h$. Therefore for every $r>0$ sufficiently small and for every $i\geq 1$ the image of the abstract Hecke algebra acting on $M_{A_{r_i},r}$ provides a finite affinoid $A_{r_i}^\circ$-algebra $\T_{A_{r_i}^\circ,r}^{\leq h}$. The morphism $w_{A_{r_i}^\circ,r}\colon\Spm\,\T_{A_{r_i}^\circ,r}^{\leq h}\to\Spm\, A_{r_i}^\circ$ is finite.
For $i<j$ we have natural inclusions $\Spm\,\T_{A_{r_j}^\circ,r}^{\leq h}\to\Spm\,\T_{A_{r_i}^\circ,r}^{\leq h}$ and corresponding restriction maps $\T_{A_{r_i}^\circ,r}^{\leq h}\to\T_{A_{r_j}^\circ,r}^{\leq h}$. We call $C_h$ the increasing union $\bigcup_{i\in\N,r>0}\Spm\,\T_{A_{r_i}^\circ,r}^{\leq h}$; it is a wide open subvariety of $C^{N}$. We denote by $\T_h$ the ring of rigid analytic functions bounded by $1$ on $C_h$. We have $\T_h=\cO(C_h)^\circ=\varprojlim_{i,r}\T_{A_{r_i}^\circ,r}^{\leq h}$. There is a natural weight map $w_h\colon C_h\to \cB_h$ that restricts to the maps $w_{A_{r_i}^\circ,r}$. It is finite because the closed ball of radius $r_h$ is adapted to $h$.

\subsection{The Galois representation associated with a family of finite slope}

Since $\cO(B_h)^\circ=\Lambda_h$, the map $w_h$ gives $\T_h$ the structure of a finite $\Lambda_h$-algebra; in particular $\T_h$ is profinite.

Let $\fm$ be a maximal ideal of $\T_h$. The residue field $k=\T_h/\fm$ is finite. 
Let Let $\T_\fm$ denote the localization of $\T_h$ at $\fm$. Since $\Lambda_h$ is henselian, 
$\T_\fm$ is a direct factor or $\T_h$, hence it is finite over $\Lambda_h$; 
it is also local noetherian and profinite. It is the ring of functions bounded by $1$ on a connected component of $C_h$.
Let $W=W(k)$ be the ring of Witt vectors of $k$. By the universal property of $W$, $\T_\fm$ is a $W$-algebra. 
$\Spm\,\T_\fm$ contains points $x$ corresponding to cuspidal eigenforms 
$f_x$ of weight $w(x)=k_x\geq 2$ and level $Np$. Let $\Q^{Np}$ be the maximal extension of $\Q$ unramified 
outside $Np$ and let $G_{\Q,Np}=\Gal(\Q^{Np}/\Q)$. The Galois representations $\rho_{f_x}$ 
associated with $f_x$ give rise to a residual representation $\overline{\rho}\colon G_{\Q,Np}\to \GL_2(k)$ 
that is independent of $f_x$. By Proposition \ref{pseudochar}, we have a pseudo-character
$$ \tau_{\T_\fm}\colon G_{\Q,Np}\to\T_\fm $$
such that for every classical point $x\colon\T_\fm\to L$, defined over some finite extension $L/\Q_p$, 
the specialization of $\tau_{\T_\fm}$ at $x$ is the trace of the usual representation $G_{\Q,Np}\to\GL_2(L)$ attached to $x$.

\begin{proposition}\label{pseudocharrep}
If $\overline{\rho}$ is absolutely irreducible there exists a unique continuous irreducible Galois representation
$$ \rho_{\T_\fm}\colon G_{\Q,Np}\to\GL_2(\T_\fm), $$
lifting $\overline{\rho}$ and whose trace is $\tau_{\T_\fm}$.
\end{proposition}

\noindent This follows from a result of Nyssen and Rouquier (\cite{nyssen}, \cite[Cor. 5.2]{rouquier}), since $\T_\fm$ is local henselian.

Let $\I^\circ$ be a finite torsion-free $\Lambda_h$-algebra. We call  \textit{family} an irreducible component of $\Spec\,\T_h$ defined by a surjective morphism $\theta\colon\T_h\to\I^\circ$ of $\Lambda_h$-algebras. Since such a map factors via $\T_\fm\to\I^\circ$ for some maximal ideal $\fm$ of $\T_h$, we can define a residual representation $\overline{\rho}$ associated with $\theta$ as above. Suppose that $\overline{\rho}$ is irreducible. By Proposition \ref{pseudocharrep} we obtain a Galois representation $\rho\colon G_\Q\to\GL_2(\I^\circ)$ associated with $\theta$.

%

\begin{remark}
If $\eta_h\notin\Q_p$, the open disc $\cB_h$ is not defined over $\Q_p$. In particular $\Lambda_h$ is not a power series ring over $\Z_p$.
\end{remark}

\subsection{Finite slope CM modular forms}\label{cmforms}

In this section we study non-ordinary finite slope CM modular forms. We say that a family is CM if all its classical points are CM.
We prove that for every positive slope $h>0$ there are no CM families with positive slope $\leq h$.
However, contrary to the ordinary case, every family of finite positive slope may contain classical CM points of weight $k\geq 2$.
Let $F$ be an imaginary quadratic field, $\ff$ an integral ideal in $F$, $I_{\ff}$ the group of fractional ideals prime to ${\ff}$.
Let $\sigma_1,\sigma_2$ be the embeddings of $F$ into $\C$ (say that $\sigma_1=\Id_F$) and let $(k_1,k_2)\in\Z^2$.
A Gr\"ossencharacter $\psi$ of infinity type $(k_1,k_2)$ defined modulo $\ff$ is a homomorphism
$\psi\colon I_{\ff}\to\C^\ast$ such that $\psi((\alpha))=\sigma_1(\alpha)^{k_1}\sigma_2(\alpha)^{k_2}$ for all $\alpha\equiv 1$ $(\mathrm{mod}^\times \ff)$ .
Consider the $q$-expansion 
$$ \sum_{\fa\subset\cO_F, (\fa,\ff)=1}\psi(\fa)q^{N(\fa)}, $$
where the sum is over ideals $\fa\subset\cO_F$ and $N(\fa)$ denotes the norm of $\fa$.
Let $F/\Q$ be an imaginary quadratic field of discriminant $D$ and let $\psi$ be a Gr\"ossencharacter of exact conductor $\ff$ and infinity type $(k-1,0)$.
By \cite[Lemma 3]{shim} the expansion displayed above defines a cuspidal newform $f(F,\psi)$ of level $N(\ff)D$.

Ribet proved in \cite[Th. 4.5]{ribet2} that if a newform $g$ of weight $k\geq 2$ and level $N$ has CM by a quadratic imaginary field $F$, one has
$g=f(F,\psi)$ for some Gr\"ossencharacter $\psi$ of $F$ of infinity type $(k-1,0)$. 

\begin{definition}\label{classcmform}
We say that a classical modular eigenform $g$ of weight $k$ and level $Np$ has CM by an imaginary quadratic field $F$
if its Hecke eigenvalues for the operators $T_\ell$, $\ell\nmid Np$, coincide with those of $f(F,\psi)$ 
for some Gr\"ossencharacter $\psi$ of $F$ of infinity type $(k-1,0)$. We also say that $g$ is CM without specifying the field.
\end{definition}

\begin{remark}
For $g$ as in the definition the Galois representations $\rho_g,\rho_{f(F,\psi)}\colon G_\Q\to\GL_2(\overline{\Q}_p)$ associated with $g$ and $f(F,\psi)$ are isomorphic, hence the image of the representation $\rho_g$ is contained in the normalizer of a torus in $\GL_2$, if and only if the form $g$ is CM. 
\end{remark}

\begin{proposition}\label{cmslopes}
Let $g$ be a CM modular eigenform of weight $k$ and level $Np^m$ with $N$ prime to $p$ and $m\geq 0$.
Then its $p$-slope is either $0$, $\frac{k-1}{2}$, $k-1$ or infinite.
\end{proposition}

\begin{proof}
Let $F$ be the quadratic imaginary field and $\psi$ the Gr\"ossencharacter of $F$ associated with the CM form $g$ by Definition \ref{classcmform}.
Let $\ff$ be the conductor of $\psi$.

We assume first that $g$ is $p$-new, so that $g=f(F,\psi)$. Let $a_p$ be the $U_p$-eigenvalue of $g$.
If $p$ is inert in $F$ we have $a_p=0$, so the $p$-slope of $g$ is infinite.
If $p$ splits in $F$ as $\fp\bar{\fp}$, then $a_p=\psi(\fp)+\psi(\bar{\fp})$. We can find an integer $n$ such that $\fp^n$ is 
a principal ideal $(\alpha)$ with $\alpha\equiv 1\,(\mathrm{mod}^\times\ff)$. Hence $\psi((\alpha))=\alpha^{k-1}$. Since $\alpha$ 
is a generator of $\fp^n$ we have $\alpha\in\fp$ and $\alpha\notin\bar{\fp}$; moreover $\alpha^{k-1}=\psi((\alpha))=\psi(\fp)^n$, 
so we also have $\psi(\fp)\in\fp-\bar{\fp}$. In the same way we find $\psi(\bar{\fp})\in\bar{\fp}-\fp$. We conclude that $\psi(\fp)+\psi(\bar{\fp})$ does not belong to $\fp$, so its $p$-adic valuation is $0$.

If $p$ ramifies as $\fp^2$ in $F$, then $a_p=\psi(\fp)$. As before we find n such that $\fp^n=(\alpha)$ with $\alpha\equiv 1\,(\mathrm{mod}^\times\ff)$. Then $(\psi(\fp))^n\psi(\fp^n)=\psi((\alpha))=\alpha^{k-1}=\fp^{n(k-1)}$. By looking at $p$-adic valuations we find that the slope is $\frac{k-1}{2}$.

%
%
If $g$ is not $p$-new, it is the $p$-stabilization of a CM form $f(F,\psi)$ of level prime to $p$. If $a_p$ is the $T_p$-eigenvalue of $f(F,\psi)$, the $U_p$-eigenvalue of $g$ is a root of the Hecke polynomial $X^2-a_pX+\zeta p^{k-1}$ for some root of unity $\zeta$. By our discussion of the $p$-new case, the valuation of $a_p$ belongs to the set $\left\{0,\frac{k-1}{2},k-1\right\}$. Then it is easy to see that the valuations of the roots of the Hecke polynomial belong to the same set.
\end{proof}

We state a useful corollary.

\begin{corollary}
There are no CM families of strictly positive slope.
\end{corollary}

\begin{proof}
We show that the eigencurve $C_h$ contains only a finite number of points corresponding to classical CM forms. It will follow that almost all classical points of a family in $C_h$ are non-CM.
Let $f$ be a classical CM form of weight $k$ and positive slope. By Proposition \ref{cmslopes} its slope is at least $\frac{k-1}{2}$. If $f$ corresponds to a point of $C_h$ its slope must be $\leq h$, so we obtain an inequality $\frac{k-1}{2}\leq h$. The set of weights $\calK$ satisfying this condition is finite.
Since the weight map $C_h\to B_h$ is finite, the set of points of $C_h$ whose weight lies in $\calK$ is finite. Hence the number of CM forms in $C_h$ is also finite. 
\end{proof}



We conclude that, in the finite positive slope case, classical CM forms can appear only as isolated points
in an irreducible component of the eigencurve $C_h$. In the ordinary case, the congruence ideal 
of a non-CM irreducible component is defined as the intersection ideal of the CM irreducible components 
with the given non-CM component. 
In the case of a positive slope family $\theta\colon\T_h\to\I^\circ$, we need to define the congruence ideal in a different way.

\subsection{Construction of the congruence ideal}\label{congrid}

Let $\theta\colon\T_h\to\I^\circ$ be a family. We write $\I=\I^\circ[p^{-1}]$.

Fix an imaginary quadratic field $F$ where $p$ is inert or ramified; let $-D$ be its discriminant.
Let $\fQ$ be a primary ideal of $\I$; then $\fq=\fQ\cap\Lambda_h$ is a primary ideal of $\Lambda_h[p^{-1}]$.
The projection $\Lambda_h\to\Lambda_h/\fq$ defines a point of $\cB_h$ (possibly non-reduced) corresponding to a weight $\kappa_\fQ\colon\Z_p^\ast\to(\Lambda_h/\fq)^\ast$.
For $r>0$ we denote by $\cB_r$ the ball of centre $1$ and radius $r$ in $\C_p$. By \cite[Prop. 8.3]{buzzard} there exist $r>0$ and a character $\kappa_{\fQ,r}\colon\Z_p^\times\cdot\cB_r\to(\Lambda_h/\fq)^\times$ extending $\kappa_\fQ$.

Let $\sigma$ be an embedding $F\into\C_p$. Let $r$ and $\kappa_{\fQ,r}$ be as above.
For $m$ sufficiently large $\sigma(1+p^m\cO_F)$ is contained in $\Z_p^\times\cdot\cB_r$, the domain of definition of $\kappa_{\fQ,r}$.

For an ideal $\ff\subset\cO_F$ let $I_{\ff}$ be the group of fractional ideals prime to ${\ff}$.
For every prime $\ell$ not dividing $Np$ we denote by $a_{\ell,\fQ}$ the image of the Hecke operator $T_\ell$ in $\I^\circ/\fQ$.
We define here a notion of non-classical CM point of $\theta$ (hence of the eigencurve $C_h$) as follows.

\begin{definition}
Let $F,\sigma,\fQ,r,\kappa_{\fQ,r}$ be as above. We say that $\fQ$ defines a CM point of weight $\kappa_{\fQ,r}$ if there exist an integer $m>0$, an ideal $\ff\subset\cO_F$ with norm $N(\ff)$ such that $DN(\ff)$ divides $N$, a quadratic extension $(\I/\fQ)^\prime$ of $\I/\fQ$ and a homomorphism $\psi\colon I_{\ff p^m}\to (\I/\fQ)^{\prime\times}$ such that:

\begin{enumerate}
\item $\sigma(1+p^m\cO_F)\subset\Z_p^\times\cdot\cB_r$;
\item for every $\alpha\in\cO_F$ with $\alpha\equiv 1\, (\mathrm{mod}^\times\ff p^m)$, $\psi((\alpha))=\kappa_{\fQ,r}(\alpha)\alpha^{-1}$;
\item $a_{\ell,\fQ}=0$ if $l$ is a prime inert in $F$ and not dividing $Np$;
\item $a_{\ell,\fQ}=\psi(\fl)+\psi(\bar{\fl})$ if $\ell$ is a prime splitting as $\fl\bar{\fl}$ in $F$ and not dividing $Np$.
\end{enumerate}
\end{definition}

\noindent Note that $\kappa_{\fQ,r}(\alpha)$ is well defined thanks to condition $1$.

\begin{remark}
If $\fP$ is a prime of $\I$ corresponding to a classical form $f$ then $\fP$ is a CM point if and only if $f$ is a CM form in the sense 
of Section \ref{cmforms}.
\end{remark}

\begin{proposition}\label{finiteCM}
The set of CM points in $\Spec\,\I$ is finite.
\end{proposition}

\begin{proof}
By contradiction assume it is infinite. Then we have an injection $\I\into\prod_\fP\I/\fP$ where $\fP$ runs over the set
of CM prime ideals of $\I$. One can assume that the imaginary quadratic field of complex multiplication is constant along $\I$. We can also assume that the ramification
of the associated Galois characters
$\lambda_{\fP}\colon G_F\to (\I/\fP)^\times$ is bounded (in support and in exponents).
On the density one set of primes of $F$ prime to $\ff p$ and of degree one, they take value in the image of $\I^\times$ hence they define 
a continuous Galois character $\lambda\colon G_F\to\I^\times$ such that $\rho_\theta=\Ind^{G_\Q}_{G_F}\lambda$, which is absurd (by specialization at non-CM classical points which do exist).
\end{proof}

\begin{definition}
The (fortuitous) congruence ideal $\fc_\theta$ associated with the family $\theta$ is defined as the intersection of all the primary ideals of $\I$ corresponding to CM points.
\end{definition}


\begin{remark}\label{cmlocus} (Characterizations of the CM locus)
\begin{enumerate}[leftmargin=*]
\item Assume that $\overline{\rho}_\theta=\Ind^{G_\Q}_{G_K}\overline{\lambda}$ for a unique imaginary quadratic field $K$.
Then the closed subscheme $V(\fc_\theta)=\Spec\,\I/\fc_\theta\subset \Spec\,\I$ is the largest subscheme on which
there is an isomorphism of Galois representations $\rho_\theta\cong\rho_\theta\otimes\left(\frac{K/\Q}{\bullet}\right)$.
Indeed, for every artinian $\Q_p$-algebra $A$, a CM point $x \colon\I\to A$ is characterized by the conditions 
$x(T_\ell)=x(T_\ell)\left(\frac{K/\Q}{\ell}\right)$ for all
primes $\ell$ prime to $Np$.
\item Note that $N$ is divisible by the discriminant $D$ of $K$. Assume that $\I$ is $N$-new and that $D$ is prime to $N/D$. 
Let $W_D$ be the Atkin-Lehner involution associated with $D$. Conjugation by $W_D$ defines an automorphism $\iota_D$
of $\T_h$ and of $\I$.
Then $V(\fc_\theta)$ coincides with the (schematic) invariant locus $(\Spec\,\I)^{\iota_D=1}$. 
\end{enumerate}
\end{remark}

\bigskip

\section{The image of the representation associated with a finite slope family}\label{imgfinslope}

It is shown by J. Lang in \cite[Th. 2.4]{lang} that, under some technical hypotheses, 
the image of the Galois representation $\rho\colon G_\Q\to\GL_2(\I^\circ)$ associated with a non-CM ordinary family 
$\theta\colon\T\to\I^\circ$ contains a congruence subgroup of $\SL_2(\I^\circ_0)$, where $\I^\circ_0$ is the subring 
of $\I^\circ$ fixed by certain ``symmetries'' of the representation $\rho$. 
In order to study the Galois representation associated with a non-ordinary family we will 
adapt some of the results in \cite{lang} to this situation. Since the crucial step (\cite[Th. 4.3]{lang}) 
requires the Galois ordinarity of the representation (as in \cite[Lemma 2.9]{hida}), the results of this section 
will not imply the existence of a congruence subgroup of $\SL_2(\I^\circ_0)$ contained in the image of $\rho$. 
However, we will prove in later sections the existence of a ``congruence Lie subalgebra'' of $\fsl_2(\I_0^\circ)$ contained in a suitably defined Lie algebra
of the image of $\rho$ by means of relative Sen theory.

For every ring $R$ we denote by $Q(R)$ its total ring of fractions.

\subsection{The group of self-twists of a family}\label{selftwists}

We follow \cite[Sec. 2]{lang} in this subsection.
Let $h\geq 0$ and $\theta\colon\T_h\to\I^\circ$ be a family of slope $\leq h$ defined over a finite torsion free $\Lambda_h$-algebra $\I^\circ$.
Recall that there is a natural map $\Lambda\to\Lambda_h$ with image $\Z_p[[\eta t]]$.

\begin{definition}
We say that $\sigma\in\Aut_{Q(\Z_p[[\eta t]])}(Q(\I^\circ))$ is a conjugate self-twist for $\theta$ if there exists a Dirichlet character $\eta_\sigma\colon G_\Q\to\I^{\circ,\times}$ such that
$$ \sigma(\theta(T_\ell))=\eta_\sigma(\ell)\theta(T_\ell) $$
for all but finitely many primes $\ell$.
\end{definition}

Any such $\sigma$ acts on $\Lambda_h=\cO_h[[t]]$ by restriction, trivially on $t$ and by a Galois automorphism on $\cO_h$.
The conjugates self-twists for $\theta$ form a subgroup of $\Aut_{Q(\Z_p[[\eta t]])}(Q(\I^\circ))$.
We recall the following result which holds without assuming the ordinarity of $\theta$.

\begin{lemma}\label{gammaabel} \cite[Lemma 7.1]{lang}
$\Gamma$ is a finite abelian $2$-group.
\end{lemma}

We suppose from now on that $\I^\circ$ is normal. The only reason for this hypothesis is that in this case $\I^\circ$ is stable under the action of $\Gamma$ on $Q(\I^\circ)$, which is not true in general. This makes it possible to define the subring $\I^\circ_0$ of elements of $\I^\circ$ fixed by $\Gamma$.

\begin{remark}
The hypothesis of normality of $\I^\circ$ is just a simplifying one.
We could work without it by introducing the $\Lambda_h$-order $\I^{\circ,\prime}=\Lambda_h[\theta(T_\ell),\ell\nmid Np]\subset\I^\circ$: this is an analogue of the $\Lambda$-order $\I^\prime$ defined in \cite[Sec. 2]{lang}, and it is stable under the action of $\Gamma$.
We would define $\I^\circ_0$ as the fixed subring of $\I^{\circ,\prime}$ and the arguments in the rest of the article could be adapted to this setting.
\end{remark}

We denote by $\cO_{h,0}$ the subring of $\cO_h$ fixed by $\Gamma$ and we put $\Lambda_{h,0}=\cO_{h,0}[[t]]$. We also denote by $K_{h,0}$ the field of fractions of $\cO_{h,0}$.

\begin{remark}\label{arithprimes}
By definition $\Gamma$ fixes $\Z_p[[\eta t]]$, so we have $\Z_p[[\eta t]]\subset\Lambda_{h,0}$. In particular it makes sense to speak about the ideal $P_k\Lambda_{h,0}$ for every arithmetic prime $P_k=(1+\eta t-u^k)\subset\Z_p[[\eta t]]$. Note that $P_k\Lambda_h$ defines a prime ideal of $\Lambda_h$ if and only if the weight $k$ belongs to the open disc $B_h$, otherwise $P_k\Lambda_h=\Lambda_h$. We see immediately that the same statement is true if we replace $\Lambda_h$ by $\Lambda_{h,0}$.
\end{remark}

Note that $\I^\circ_0$ is a finite extension of $\Lambda_{h,0}$ because $\I^\circ$ is a finite $\Lambda_h$-algebra.
Moreover, we have $K_h^\Gamma=K_{h,0}$ (although the inclusion $\Lambda_h\cdot\I_0^\circ\subset \I^\circ$ 
may not be an equality).

We define two open normal subgroups of $G_\Q$ by:

\begin{itemize}
\item $H_0=\bigcap_{\sigma\in\Gamma}\ker\eta_\sigma$;
\item $H=H_0\cap\ker(\det\overline{\rho})$.
\end{itemize}

\noindent Note that $H_0$ is an open normal subgroup of $G_\Q$ and that $H$ is
a pro-$p$ open normal subgroup of $H_0$ and of $G_\Q$.

\subsection{The level of a general ordinary family}

We recall the main result of \cite{lang}. Denote by $\T$ the big ordinary Hecke algebra, which is finite over $\Lambda=\Z_p[[T]]$. Let $\theta\colon\T\to\I^\circ$ be an ordinary family with associated Galois representation $\rho\colon G_\Q\to\GL_2(\I^\circ)$. The representation $\rho$ is $p$-ordinary, which means that its restriction $\rho\vert_{D_p}$ to a decomposition subgroup $D_p\subset G_\Q$ is reducible. There exist two characters $\varepsilon,\delta\colon D_p\to\I^{\circ,\times}$, with $\delta$ unramified, such that $\rho\vert_{D_p}$ is an extension of $\varepsilon$ by $\delta$.

Denote by $\F$ the residue field of $\I^\circ$ and by $\overline{\rho}$ the representation $G_\Q\to\GL_2(\F)$ obtained by reducing $\rho$ modulo the maximal ideal of $\I^\circ$. Lang introduces the following technical condition.

\begin{definition}
The $p$-ordinary representation $\overline{\rho}$ is called $H_0$-regular if $\overline{\varepsilon}\vert_{D_p\cap H_0}\neq\overline{\delta}\vert_{D_p\cap H_0}$.
\end{definition}

The following result states the existence of a Galois level for $\rho$.

\begin{theorem}\label{ordlevel} \cite[Th. 2.4]{lang}
Let $\rho\colon G_\Q\to\GL_2(\I^\circ)$ be the representation associated with an ordinary, non-CM family $\theta\colon \T\to\I^\circ$.
Assume that $p>2$, the cardinality of $\F$ is not $3$ and the residual representation $\overline{\rho}$ 
is absolutely irreducible and $H_0$-regular. Then there exists $\gamma\in\GL_2(\I^\circ)$ such that $\gamma\cdot \im\,\rho\cdot \gamma^{-1}$ contains a congruence subgroup of $\SL_2(\I^\circ_0)$.
\end{theorem}

\noindent The proof relies on the analogous result proved by Ribet (\cite{ribet1}) and Momose (\cite{momose}) 
for the $p$-adic representation associated with a classical modular form.

\subsection{An approximation lemma}\label{sectapprox}

In this subsection we prove an analogue of \cite[Lemma 4.5]{hidatil}.
It replaces in our approach the use of Pink's Lie algebra theory, which is relied upon in the case of ordinary representations
in \cite{hida} and \cite{lang}.
Let $\I_0^\circ$ be a domain that is finite torsion free over $\Lambda_h$. It does not need to be related to a Hecke algebra for the moment. 

Let $N$ be an open normal subgroup of $G_\Q$ and let $\rho\colon N\to\GL_2(\I_0^\circ)$ be an arbitrary continuous representation.
We denote by $\fm_{\I_0^\circ}$ the maximal ideal of $\I_0^\circ$ and by $\F=\I_0^\circ/\fm_{\I_0^\circ}$ its residue field of cardinality $q$.
In the lemma we do not suppose that $\rho$ comes from a family of modular forms.
We will only assume that it satisfies the following technical condition:

\begin{definition}\label{Z_p-reg}
Keep notations as above.
We say that the representation $\rho\colon N\to\GL_2(\I_0^\circ)$ is $\Z_p$-regular if there exists $d\in\im\,\rho$ with 
eigenvalues $d_1,d_2\in\Z_p$ such that $d_1^2\not\equiv d_2^2\pmod{p}$.
We call $d$ a $\Z_p$-regular diagonal element.
If $N^\prime$ is an open normal subgroup of $N$ then we say that $\rho$ is $(N^\prime,\Z_p)$-regular if $\rho\vert_{N^\prime}$ is $\Z_p$-regular.
\end{definition}

Note that $\rho(\delta)\in\im\,\rho$ is of finite order dividing $p-1$.
Let $B^{\pm}$ denote the Borel subgroups consisting of upper, respectively lower, triangular matrices in $\GL_2$.
Let $U^{\pm}$ be the unipotent radical of $B^{\pm}$.

\begin{proposition}\label{approx}
Let $\I_0^\circ$ be a finite torsion free $\Lambda_{h,0}$-algebra, $N$ an open normal subgroup of $G_\Q$ and $\rho$ a continuous representation $N\to\GL_2(\I_0^\circ)$ that is $\Z_p$-regular. Suppose (upon replacing $\rho$ by a conjugate) that the $\Z_p$-regular element is diagonal.
Let $\bP$ be an ideal of $\I_0^\circ$ and $\rho_\bP\colon N\to\GL_2(\I_0^\circ/\bP)$ be the representation given by the reduction of $\rho$ modulo $\bP$. Let $U^\pm(\rho)$, respectively $U^\pm(\rho_\bP)$ be the upper and lower unipotent subgroups of the image $\im\,\rho$, respectively $\im\,\rho_\bP$.
Then the natural maps $U^+(\rho_\theta)\to U^+(\rho_\bP)$ and $U^-(\rho_\theta)\to U^-(\rho_\bP)$ are surjective.
\end{proposition}

\begin{remark}
The ideal $\bP$ in the proposition is not necessarily prime. At a certain point we will need to take $\bP=P\I_0^\circ$ for a prime ideal $P$ of $\Lambda_{h,0}$.
\end{remark}


As in \cite[Lemma 4.5]{hidatil} we need two lemmas.
Since the argument is the same for $U^+$ and $U^-$,
we will only treat here the upper triangular case $U=U^+$ and $B=B^+$.

For $\ast=U, B$ and every $j\geq 1$ we define the groups 
$$ \Gamma_{\ast}(\bP^j)=\{x\in \SL_2(\I_0^\circ)\, |\, x\,\, (\mathrm{mod}\,\,\bP^j)\in \ast (\I_0^\circ/\bP^j)\}.$$
Let $\Gamma_{\I_0^\circ}(\bP^j)$
be the kernel of the reduction morphism $\pi_j\colon \SL_2(\I_0^\circ)\to \SL_2(\I_0^\circ/\bP^j)$.
Note that $\Gamma_{U}(\bP^j)=\Gamma_{\I_0^\circ}(\bP^j)U(\I_0^\circ)$ consists of matrices 
$\left(\begin{array}{cc}a&b\\c&d\end{array}\right)$ such that $a,d\equiv 1 \pmod{\bP^j}$, $c\equiv 0\pmod{\bP^j}$.
Let $K=\im\,\rho$ and
$$ K_{U}(\bP^j)=K\cap\Gamma_{U}(\bP^j),\quad
K_{B}(\bP^j)=K\cap\Gamma_{B}(\bP^j). $$

Since $U(\I_0^\circ)$ and $\Gamma_{\I_0^\circ}(\bP)$
are $p$-profinite, the groups $\Gamma_{U}(\bP^j)$ and $K_{U}(\bP^j)$ for all $j\geq 1$ are also $p$-profinite.
Note that
$$ \left[\left(\begin{smallmatrix}a&b\\
c&-a\end{smallmatrix}\right),\left(\begin{smallmatrix}e&f\\
g&-e\end{smallmatrix}\right)\right]
=\left(\begin{smallmatrix}bg-cf&2(af-be)\\
2(ce-ag)&cf-bg\end{smallmatrix}\right). $$
\noindent From this we obtain the following.

\begin{lemma}\label{ijk}
If $X,Y\in\fsl_2(\I_0^\circ)\cap\left(\begin{smallmatrix}\bP^j&\bP^k\\
\bP^i&\bP^j\end{smallmatrix}\right)$  with  $i\ge j\ge k$, then 
$[X,Y]\in \left(\begin{smallmatrix}\bP^{i+k}&\bP^{j+k}\\
\bP^{i+j}&\bP^{i+k}\end{smallmatrix}\right)$.
\end{lemma}

We denote by $\mathrm{D}\Gamma_{U}(\bP^j)$ the topological commutator subgroup $(\Gamma_{U}(\bP^j),\Gamma_{U}(\bP^j))$. Lemma \ref{ijk} tells us that
\begin{equation}\label{BN} \mathrm{D}\Gamma_{U}(\bP^j)\subset\Gamma_{B}(\bP^{2j})\cap\Gamma_{U}(\bP^j). \end{equation}

By the $\Z_p$-regularity assumption, there exists a diagonal element $d\in K$ whose eigenvalues are in $\Z_p$ and distinct modulo $p$.
Consider the element $\delta=\lim_{n\to\infty}d^{p^n}$, which belongs to $K$ since this is $p$-adically complete.
In particular $\delta$ normalizes $K$. It is also diagonal with coefficients in $\Z_p$, so it normalizes $K_{U}(\bP^j)$ and $\Gamma_{B}(\bP^j)$. Since $\delta^p=\delta$, the eigenvalues $\delta_1$ and $\delta_2$ of $\delta$ are roots of unity of order dividing $p-1$. They still satisfy $\delta_1^2\neq\delta_2^2$ as $p\neq 2$. 

Set $\alpha=\delta_1/\delta_2\in\F_p^\times$ and let $a$ be the order of $\alpha$ as a root of unity. 
We see $\alpha$ as an element of $\Z_p^\times$ via the Teichm\"uller lift. 
Let $H$ be a $p$-profinite group normalized by $\delta$. 
Since $H$ is $p$-profinite, every $x\in H$ has a unique $a$-th root. 
We define a map $\Delta\colon H\to H$ given by 
$$ \Delta(x)=[x\cdot \ad(\delta) (x)^{\alpha^{-1}}\cdot\ad(\delta^2)(x)^{\alpha^{-2}}\cdot\,\cdots\,\cdot
\ad(\delta^{a-1})(x)^{\alpha^{1-a}}]^{1/a} $$

\begin{lemma}\label{2jlem}
If $u\in \Gamma_{U}(\bP^j)$ for some $j\ge 1$,
then $\Delta^2(u)\in \Gamma_{U}(\bP^{2j})$
and $\pi_j(\Delta(u))=\pi_j(u)$.
\end{lemma}

\begin{proof}
If $u\in \Gamma_{U}(\bP^j)$, we have $\pi_j(\Delta(u))=\pi_j(u)$
as $\Delta$ is the identity map on $U(\I_0^\circ/\bP^j)$.
Let $\mathrm{D}\Gamma_{U}(\bP^j)$  be the topological commutator subgroup of  $\Gamma_{U}(\bP^j)$.
Since $\Delta$ induces the projection of the $\Z_p$-module $\Gamma_{U}(\bP^j)/\mathrm{D}\Gamma_{U}(\bP^j)$ onto its $\alpha$-eigenspace for $\ad(d)$,
it is a projection onto $U(\I_0^\circ) \mathrm{D}\Gamma_{U}(\bP^j)/\mathrm{D}\Gamma_{U}(\bP^j)$.
The fact that this is exactly the $\alpha$-eigenspace comes from the Iwahori 
decomposition of $\Gamma_{U}(\bP^j)$, hence
a similar direct sum decomposition holds in the abelianization $\Gamma_{U}(\bP^j)/\mathrm{D}\Gamma_{U}(\bP^j)$.

By \eqref{BN}, $\mathrm{D}\Gamma_{U}(\bP^j)\subset \Gamma_{B}(\bP^{2j})\cap \Gamma_{U}(\bP^j)$.
Since the $\alpha$-eigenspace
of $\Gamma_{U}(\bP^j)/\mathrm{D}\Gamma_{U}(\bP^j)$ is inside  $\Gamma_{B}(\bP^{2j})$,
$\Delta$ projects $u\Gamma_{U}(\bP^j)$ to
$$ \overline{\Delta}(u)\in (\Gamma_{B}(\bP^{2j})\cap \Gamma_{U}(\bP^j))/\mathrm{D}\Gamma_{U}(\bP^j). $$
In particular, $\Delta(u)\in \Gamma_{B}(\bP^{2j})\cap \Gamma_{U}(\bP^j)$.
Again apply  $\Delta$. Since $\Gamma_{B}(\bP^{2j})/\Gamma_{\I_0^\circ}(\bP^{2j})$ is sent
to $\Gamma_{U}(\bP^{2j})/\Gamma_{\I_0^\circ}(\bP^{2j})$ by $\Delta$,
we get $\Delta^2(u)\in \Gamma_{U}(\bP^{2j})$ as desired.
\end{proof}

%

\begin{proof}
We can now prove Proposition \ref{approx}.
Let $\overline{u}\in U(\I_0^\circ/\bP)\cap\im(\rho_\bP)$.
Since the reduction map $\im(\rho)\to\im(\rho_\bP)$ induced by $\pi_1$ 
is surjective, there exists $v \in\im(\rho)$
such that $\pi_1(v)=\overline{u}$.
Take $u_1\in U(\I_0^\circ)$ such that $\pi_1(u_1)=\overline{u}$ (this is possible since $\pi_1\colon U(\Lambda_h)\to U(\Lambda_h/P)$ is surjective).
Then $v u_1^{-1}\in\Gamma_{\I_0^\circ}(\bP)$,
so $v\in K_{U}(\bP)$.

By compactness of $K_{U}(\bP)$ and by Lemma \ref{2jlem}, starting with $v$ as above, we see that
$\lim_{m\to \infty}\Delta^m(v)$ converges $\bP$-adically 
to  $\Delta^\infty(v)\in U(\I_0^\circ)\cap K$
with $\pi_1(\Delta^\infty(v))=\overline{u}$.
\end{proof}

\begin{remark}
Proposition \ref{approx} is true with the same proof if we replace $\Lambda_{h,0}$ by $\Lambda_h$ and $\I_0^\circ$ by a finite torsion free $\Lambda_h$-algebra.
\end{remark}

As a first application of Proposition \ref{approx} we give a result that we will need in the next subsection.
Given a representation $\rho\colon G_\Q\to\GL_2(\I^\circ)$ and every ideal $\bP$ of $\I^\circ$ we define $\rho_\bP$, $U^{\pm}(\rho)$ and $U^{\pm}(\rho_\bP)$ as above, by replacing $\I_0^\circ$ by $\I^\circ$.

\begin{proposition}\label{nontrivunip}
Let $\theta\colon\T_h\to\I^\circ$ be a family of slope $\leq h$ and $\rho_\theta\colon G_\Q\to\GL_2(\I^\circ)$ 
be the representation associated with $\theta$.
Suppose that $\rho_\theta$ is $(H_0,\Z_p)$-regular and let $\rho$ be a conjugate of $\rho_\theta$ such that $\im\,\rho\vert_{H_0}$ contains a diagonal $\Z_p$-regular element. 
Then $U^+(\rho)$ and $U^-(\rho)$ are both nontrivial.
\end{proposition}


\begin{proof}
By density of classical points in $\T_h$ we can choose a prime ideal $\bP\subset\I^\circ$ corresponding to a classical modular form $f$.
The modulo $\bP$ representation $\rho_\bP$ is the $p$-adic representation classically associated with $f$.
By the results of \cite{ribet1} and \cite{momose}, there exists an ideal $\fl_\bP$ of $\Z_p$ such that $\im\,\rho_\bP$ contains the congruence subgroup $\Gamma_{\Z_p}(\fl_\bP)$.
In particular $U^+(\rho_\bP)$ and $U^-(\rho_\bP)$ are both nontrivial.
Since the maps $U^+(\rho)\to U^+(\rho_\bP)$ and $U^-(\rho)\to U^-(\rho_\bP)$ are surjective we find nontrivial elements in $U^+(\rho)$ and $U^-(\rho)$.
\end{proof}

We adapt the work in \cite[Sec. 7]{lang} to show the following.

\begin{proposition}\label{propI0}
Suppose that the representation $\rho\colon G_\Q\to\GL_2(\I^\circ)$ is $(H_0,\Z_p)$-regular. 
Then there exists $g\in\GL_2(\I^\circ)$ such that the conjugate representation $g\rho g^{-1}$ satisfies the following two properties:
\begin{enumerate}
\item the image of $g\rho g^{-1}\vert_{H_0}$ is contained in $\GL_2(\I^\circ_0)$;
\item the image of $g\rho g^{-1}\vert_{H_0}$ contains a diagonal $\Z_p$-regular element.
\end{enumerate}
\end{proposition}

\begin{proof}
As usual we choose a $\GL_2(\I^\circ)$-conjugate of $\rho$ such that a $\Z_p$-regular element $d$ is diagonal.
We still write $\rho$ for this conjugate representation and we show that it also has property (1).

Recall that for every $\sigma\in\Gamma$ there is a character $\eta_\sigma\colon G_\Q\to(\I^\circ)^\times$ and an equivalence $\rho^\sigma\cong\rho\otimes\eta_\sigma$.
Then for every $\sigma\in\Gamma$ there exists $\bt_\sigma\in\GL_2(\I^\circ)$ such that, for all $g\in G_\Q$,
\begin{equation}\label{rhoequiv}
\rho^\sigma(g)=\bt_\sigma\eta_\sigma(g)\rho(g)\bt_\sigma^{-1}.
\end{equation}

We prove that the matrices $\bt_\sigma$ are diagonal.
Let $\rho(t)$ be a non-scalar diagonal element in $\im\,\rho$ (for example $d$).
Evaluating (\ref{rhoequiv}) at $g=t$ we find that $\bt_\sigma$ must be either a diagonal or an antidiagonal matrix.
Now by Proposition \ref{nontrivunip} there exists a nontrivial element $\rho(u^+)\in\im\,\rho\cap U^+(\I^\circ)$.
Evaluating (\ref{rhoequiv}) at $g=u^+$ we find that $\bt_\sigma$ cannot be antidiagonal.

It is shown in \cite[Lemma 7.3]{lang} that there exists an extension $A$ of $\I^\circ$, at most quadratic, and a function $\zeta\colon\Gamma\to A^\times$ such that $\sigma\to\bt_\sigma\zeta(\sigma)^{-1}$ defines a cocycle with values in $\GL_2(A)$. 
The proof of this result does not require the ordinarity of $\rho$. Eq. (\ref{rhoequiv}) remains true if we replace $\bt_\sigma$ with $\bt_\sigma\zeta(\sigma)^{-1}$, so we can and do suppose from now on that $\bt_\sigma$ is a cocycle with values in $\GL_2(A)$. In the rest of the proof we assume for simplicity that $A=\I^\circ$, but everything works in the same way if $A$ is a quadratic extension of $\I^\circ$ and $\F$ is the residue field of $A$.

Let $V=(\I^\circ)^2$ be the space on which $G_\Q$ acts via $\rho$. As in \cite[Sec. 7]{lang} we use the cocycle $\bt_\sigma$ to define a twisted action of $\Gamma$ on $(\I^\circ)^2$.
For $v=(v_1,v_2)\in V$ we denote by $v^\sigma$ the vector $(v_1^\sigma,v_2^\sigma)$. 
We write $v^{[\sigma]}$ for the vector $\bt_\sigma^{-1}v^\sigma$.
Then $v\to v^{[\sigma]}$ gives an action of $\Gamma$ since $\sigma\mapsto \bt_\sigma$ is a cocycle. Note that this action is $\I^\circ_0$-linear.

Since $\bt_\sigma$ is diagonal for every $\sigma\in\Gamma$, the submodules $V_1=\I^\circ(1,0)$ and $V_2=\I^\circ(0,1)$ are stable under the action of $\Gamma$.
We show that each $V_i$ contains an element fixed by $\Gamma$.
We denote by $\F$ the residue field $\I^\circ/\fm_\I^\circ$.
Note that the action of $\Gamma$ on $V_i$ induces an action of $\Gamma$ on the one-dimensional $\F$-vector space $V_i\otimes\I^\circ/\fm_{\I^\circ}$.
We show that for each $i$ the space $V_i\otimes\I^\circ/\fm_{\I^\circ}$ contains a nonzero element $\overline{v}_i$ fixed by $\Gamma$.
This is a consequence of the following argument, a form of which appeared in an early preprint of \cite{lang}.
Let $w$ be any nonzero element of $V_i\otimes\I^\circ/\fm_{\I^\circ}$ and let $a$ be a variable in $\F$. The sum
$$ S_{aw}=\sum_{\sigma\in\Gamma}(aw)^{[\sigma]} $$
is clearly $\Gamma$-invariant. We show that we can choose $a$ such that $S_{aw}\neq 0$. Since $V_i\otimes\I^\circ/\fm_{\I^\circ}$ is one-dimensional, for every $\sigma\in\Gamma$ there exists $\alpha_\sigma\in\F$ such that $w^{[\sigma]}=\alpha_\sigma w$. Then
$$ S_{aw}=\sum_{\sigma\in\Gamma}(aw)^{[\sigma]}=\sum_{\sigma\in\Gamma}a^\sigma w^{[\sigma]}=\sum_{\sigma\in\Gamma}a^\sigma\alpha_\sigma w=\left(\sum_{\sigma\in\Gamma}a^\sigma\alpha_\sigma a^{-1}\right)aw. $$
By Artin's lemma on the independence of characters, the function $f(a)=\sum_{\sigma\in\Gamma}a^\sigma\alpha_\sigma a^{-1}$ cannot be identically zero on $\F$. By choosing a value of $a$ such that $f(a)\neq 0$ we obtain a nonzero element $\overline{v}_i=S_{aw}$ fixed by $\Gamma$.

We show that $\overline{v}_i$ lifts to an element $v_i\in V_i$ fixed by $\Gamma$. 
Let $\sigma_0\in\Gamma$. By Lemma \ref{gammaabel} $\Gamma$ is a finite abelian $2$-group, so the minimal polynomial $P_m(X)$ of $[\sigma_0]$ acting on $V_i$ divides $X^{2^k}-1$ for some integer $k$. In particular the factor $X-1$ appears with multiplicity at most $1$. We show that its multiplicity is exactly $1$.
If $\overline{P_m}$ is the reduction of $P_m$ modulo $\fm_{\I^\circ}$ then $\overline{P_m}([\sigma_0])=0$ on $V_i\otimes\I^\circ/\fm_{\I^\circ}$.
By our previous argument there is an element of $V_i\otimes\I^\circ/\fm_{\I^\circ}$ fixed by $\Gamma$ (hence by $[\sigma_0]$) so we have $(X-1)\mid\overline{P_m(X)}$. Since $p>2$ the polynomial $X^{2^k}-1$ has no double roots modulo $\fm_{\I^\circ}$, so neither does $\overline{P_m}$. By Hensel's lemma the factor $X-1$ lifts to a factor $X-1$ in $P_m$ and $\overline{v}_i$ lifts to an element $v_i\in V_i$ fixed by $[\sigma_0]$. Note that $\I^\circ\cdot v_i=V_i$ by Nakayama's lemma since $\overline{v}_i\neq 0$. 

We show that $v_i$ is fixed by all of $\Gamma$.
Let $W_{[\sigma_0]}=\I^\circ v_i$ be the one-dimensional eigenspace for $[\sigma_0]$ in $V_i$. Since $\Gamma$ is abelian $W_{[\sigma_0]}$ is stable under $\Gamma$.
Let $\sigma\in\Gamma$. Since $\sigma$ has order $2^k$ in $\Gamma$ for some $k\geq 0$ and $v_i^{[\sigma]}\in W_{[\sigma_0]}$, there exists a root of unity $\zeta_\sigma$ of order $2^k$ satisfying $v_i^{[\sigma]}=\zeta_\sigma v_i$. Since $\overline{v}_i^{[\sigma]}=\overline{v}_i$, the reduction of $\zeta_\sigma$ modulo $\fm_{\I^\circ}$ must be $1$. As before we conclude that $\zeta_\sigma=1$ since $p\neq 2$. 

We found two elements $v_1\in V_1$, $v_2\in V_2$ fixed by $\Gamma$. We show that every element of $v\in V$ fixed by $\Gamma$ must belong to the $\I^\circ_0$-submodule generated by $v_1$ and $v_2$. We proceed as in the end of the proof of \cite[Th. 7.5]{lang}. 
Since $V_1$ and $V_2$ are $\Gamma$-stable we must have $v\in V_1$ or $v\in V_2$. Suppose without loss of generality that $v\in V_1$. Then $v=\alpha v_1$ for some $\alpha\in\I^\circ$. If $\alpha\in\I^\circ_0$ then $v\in\I^\circ_0 v_1$, as desired. If $\alpha\notin\I^\circ_0$ then there exists $\sigma\in\Gamma$ such that $\alpha^\sigma\neq\alpha$. Since $v$ is $[\sigma]$-invariant we obtain $(\alpha v_1)^{[\sigma]}=\alpha^\sigma v_1^{[\sigma]}=\alpha^\sigma v_1\neq \alpha v$, so $\alpha v_1$ is not fixed by $[\sigma]$, a contradiction.

Now $(v_1,v_2)$ is a basis for $V$ over $\I^\circ$, so the $\I^\circ_0$ submodule $V_0=\I^\circ_0v_1+\I^\circ_0v_2$ is an $\I^\circ_0$-lattice in $V$.
Recall that $H_0=\bigcap_{\sigma\in\Gamma}\ker\eta_\sigma$.
We show that $V_0$ is stable under the action of $H_0$ via $\rho\vert_{H_0}$, i.e. that if $v\in V$ is fixed by $\Gamma$, so is $\rho(h)v$ for every $h\in H_0$. 
This is a consequence of the following computation, where $v$ and $h$ are as before and $\sigma\in\Gamma$:
$$ (\rho(h)v)^{[\sigma]}=\bt_\sigma^{-1}\rho(h)^\sigma v^\sigma=\bt_\sigma^{-1}\eta_\sigma(h)\rho(h)^\sigma v^\sigma=\bt_\sigma^{-1}\bt_\sigma\rho(h)\bt_\sigma^{-1}v^\sigma=\rho(h)v^{[\sigma]}. $$

Since $V_0$ is an $\I^\circ_0$-lattice in $V$ stable under $\rho\vert_{H_0}$, we conclude that $\im\,\rho\vert_{H_0}\subset\GL_2(\I^\circ_0)$.
\end{proof}

\subsection{Fullness of the unipotent subgroups}

From now on we write $\rho$ for the element in its $\GL_2(\I^\circ)$ conjugacy class such that $\rho\vert_{H_0}\in\GL_2(\I^\circ_0)$.
Recall that $H$ is the open subgroup of $H_0$ defined by the condition $\det\overline{\rho}(h)=1$ for every $h\in H$.
As in \cite[Sec. 4]{lang} we define a representation $H\to\SL_2(\I^\circ_0)$ by
$$ \rho_0=\rho\vert_H\otimes\left(\det\rho\vert_H\right)^{-\frac{1}{2}}. $$

We can take the square root of the determinant thanks to the definition of $H$.
We will use the results of \cite{lang} to deduce that the $\Lambda_{h,0}$-module generated by the unipotent subgroups of the image of $\rho_0$ is big.
We will later deduce the same for $\rho$.

%

We fix from now on a height one prime $P\subset\Lambda_{h,0}$ with the following properties:
\begin{enumerate}
\item there is an arithmetic prime $P_k\subset\Z_p[[\eta t]]$ satisfying $k>h+1$ and $P=P_k\Lambda_{h,0}$;
\item every prime $\fP\subset\I^\circ$ lying above $P$ corresponds to a non-CM point.
\end{enumerate}
Such a prime always exists. Indeed, by Remark \ref{arithprimes} every classical weight $k>h+1$ contained in the disc $B_h$ defines a prime $P=P_k\Lambda_{h,0}$ satisfying (1), so such primes are Zariski-dense in $\Lambda_{h,0}$, while the set of CM primes in $\I^\circ$ is finite by Proposition \ref{finiteCM}.

\begin{remark}\label{etale}
Since $k>h+1$, every point of $\Spec\,\T_h$ above $P_k$ is classical by \cite[Th. 6.1]{coleman}. Moreover the weight map is \'etale at every 
such point by \cite[Th. 11.10]{kisin}. In particular the prime $P\I_0^\circ=P_k\I_0^\circ$ splits as a product of distinct primes of $\I_0^\circ$.
\end{remark}

Make the technical assumption that the order of the residue field $\F$ of $\I^\circ$ is not $3$.
For every ideal $\bP$ of $\I^\circ_0$ over $P$ we let $\pi_\bP$ be the projection $\SL_2(\I^\circ_0/\bP)\to\SL_2(\I^\circ_0/\bP)$.
We still denote by $\pi_\bP$ the restricted maps $U^\pm(\I^\circ_0/\bP)\to U^\pm(\I^\circ_0/\bP)$.

Let $G=\im\,\rho_0$. For every ideal $\bP$ of $\I^\circ_0$ we denote by $\rho_{0,\bP}$ the representation $\pi_\bP(\rho_0)$ and by $G_\bP$ the image of $\rho_\bP$. Clearly $G_\bP=\pi_\bP(G)$.
We state two results from Lang's work that come over unchanged to the non-ordinary setting.

\begin{proposition}\label{openproj}\cite[Cor. 6.3]{lang}
Let $\fP$ be a prime of $\I^\circ_0$ over $P$. Then $G_{\fP}$ contains a congruence subgroup $\Gamma_{\I^\circ_0/\fP}(\fa)\subset\SL_2(\I^\circ_0/\fP)$. In particular $G_\fP$ is open in $\SL_2(\I^\circ_0/\fP)$.
\end{proposition}

\begin{proposition}\label{openprod}\cite[Prop. 5.1]{lang}
Assume that for every prime $\fP\subset\I^\circ_0$ over $P$ the subgroup $G_\fP$ is open in $\SL_2(\I^\circ_0/\fP)$.
Then the image of $G$ in $\prod_{\fP|P}\SL_2(\I^\circ_0/\fP)$ through the map $\prod_{\fP|P}\pi_\fP$ contains a product of congruence subgroups $\prod_{\fP|P}\Gamma_{\I^\circ_0/\fP}(\fa_\fP)$.
\end{proposition}

\begin{remark}
The proofs of Propositions \ref{openproj} and \ref{openprod} rely on the fact that the big ordinary Hecke algebra is \'etale over $\Lambda$ at every arithmetic point. In order for these proofs to adapt to the non-ordinary setting it is essential that the prime $P$ satisfies the properties above Remark \ref{etale}.
\end{remark}

We let $U^\pm(\rho_0)=G\cap U^{\pm}(\I^\circ_0)$ and $U^\pm(\rho_\bP)=G_\bP\cap U^{\pm}(\I^\circ_0/\bP)$.
We denote by $U(\rho_\bP)$ either the upper or lower unipotent subgroups of $G_{\bP}$ (the choice will be fixed throughout the proof).
By projecting to the upper right element we identify $U^+(\rho_0)$ with a $\Z_p$-submodule of $\I^\circ_0$ and $U^+(\rho_{0,\bP})$ with a $\Z_p$-submodule of $\I^\circ_0/\bP$. We make analogous identifications for the lower unipotent subgroups.
We will use Proposition \ref{openprod} and Proposition \ref{approx} to show that, for both signs, $U^{\pm}(\rho)$ spans $\I^\circ_0$ over $\Lambda_{h,0}$.

First we state a version of \cite[Lemma 4.10]{lang}, with the same proof. Let $A$ and $B$ be Noetherian rings with $B$ integral over $A$. We call $A$-\textit{lattice} an $A$-submodule of $B$ generated by the elements of a basis of $Q(B)$ over $Q(A)$.

\begin{lemma}\label{lattice}
Any $A$-lattice in $B$ contains a nonzero ideal of $B$. Conversely, every nonzero ideal of $B$ contains an $A$-lattice.
\end{lemma}

We prove the following proposition by means of Proposition \ref{approx}. We could also use Pink theory as in \cite[Sec. 4]{lang}.

\begin{proposition}\label{lambdaspan_0}
Consider $U^\pm(\rho_0)$ as subsets of $Q(\I^\circ_0)$. For each choice of sign the $Q(\Lambda_{h,0})$-span of $U^\pm(\rho_0)$ is $Q(\I^\circ_0)$. Equivalently the $\Lambda_{h,0}$-span of $U^\pm(\rho_0)$ contains a $\Lambda_{h,0}$-lattice in $\I^\circ_0$.
\end{proposition}

\begin{proof}
Keep notations as above. We omit the sign when writing unipotent subgroups and we refer to either the upper or lower ones (the choice is fixed throughout the proof).
Let $P$ be the prime of $\Lambda_{h,0}$ chosen above.
By Remark \ref{etale} the ideal $P\I_0^\circ$ splits as a product of distinct primes in $\I_0^\circ$.
When $\fP$ varies among these primes, the map $\bigoplus_{\fP|P}\pi_\fP$ gives embeddings of $\Lambda_{h,0}/P$-modules $\I^\circ_0/P\I^\circ_0\into\bigoplus_{\fP|P}\I^\circ_0/\fP$ and $U(\rho_{P\I_0^\circ})\into\bigoplus_{\fP|P}U(\rho_\fP)$.
The following diagram commutes:
\begin{equation}\label{immersions}
\begin{tikzcd}[baseline=(current bounding box.center)]
U(\rho_{P\I_0^\circ}) \arrow[hook]{r}{\bigoplus_{\fP|P}\pi_\fP}\arrow[hook]{d}
& \bigoplus_{\fP|P}U(\rho_\fP)\arrow[hook]{d}\\
\I^\circ_0/P\I^\circ_0 \arrow[hook]{r}{\bigoplus_{\fP|P}\pi_\fP}
& \bigoplus_{\fP|P}\I^\circ_0/\fP
\end{tikzcd}
\end{equation}

By Proposition \ref{openprod} there exist ideals $\fa_\fP\subset\I^\circ_0/\fP$ such that $(\bigoplus_{\fP|P}\pi_\fP)(G_{P\I_0^\circ})\supset\bigoplus_{\fP|P}\Gamma_{\I^\circ_0/\fP}(\fa_\fP)$.
In particular $(\bigoplus_{\fP|P}\pi_\fP)(U(\rho_{P\I_0^\circ}))\supset\bigoplus_{\fP|P}(\fa_\fP)$.
By Lemma \ref{lattice} each ideal $\fa_\fP$ contains a basis of $Q(\I^\circ_0/\fP)$ over $Q(\Lambda_{h,0}/P)$, so that the $Q(\Lambda_{h,0}/P)$-span of $\bigoplus_{\fP|P}\fa_\fP$ is the whole $\bigoplus_{\fP|P}Q(\I^\circ_0/\fP)$.
Then the $Q(\Lambda_{h,0}/P)$-span of $(\bigoplus_{\fP|P}\pi_\fP)(G_\fP\cap U(\rho_{\fP}))$ is also $\bigoplus_{\fP|P}Q(\I^\circ_0/\fP)$.
By commutativity of diagram (\ref{immersions}) we deduce that the $Q(\Lambda_{h,0}/P)$-span of $G_P\cap U(\rho_{P\I_0^\circ})$ is $Q(\I^\circ_0/P\I^\circ_0)$.
In particular $G_{P\I_0^\circ}\cap U(\rho_{P\I_0^\circ})$ contains a $\Lambda_{h,0}/P$-lattice, hence by Lemma \ref{lattice} a nonzero ideal $\fa_P$ of $\I^\circ_0/P\I^\circ_0$.

Note that the representation $\rho_0\colon H\to\SL_2(\I_0^\circ)$ satisfies the hypotheses of Proposition \ref{approx}.
Indeed we assumed that $\rho\colon G_\Q\to\GL_2(\I)$ is $(H_0,\Z_p)$-regular, so the image of $\rho\vert_{H_0}$ contains a diagonal $\Z_p$-regular element $d$.
Since $H$ is a normal subgroup of $H_0$, $\rho(H)$ is a normal subgroup of $\rho(H_0)$ and it is normalized by $d$.
By a trivial computation we see that the image of $\rho_0=\rho\vert_H\otimes(\det\rho\vert_H)^{-1/2}$ is also normalized by $d$.

Let $\fa$ be an ideal of $\I^\circ_0$ projecting to $\fa_P\subset U(\rho_{0,P\I_0^\circ})$. By Proposition \ref{approx} applied to $\rho_0$ we obtain that the map $U(\rho_0)\to U(\rho_{0,P\I_0^\circ})$ is surjective, so the $\Z_p$-module $\fa\cap U(\rho_0)$ also surjects to $\fa_P$.
Since $\Lambda_{h,0}$ is local we can apply Nakayama's lemma to the $\Lambda_{h,0}$-module $\Lambda_{h,0}(\fa\cap U(\rho_0)$ to conclude that it coincides with $\fa$. Hence $\fa\subset\Lambda_{h,0}\cdot U(\rho_0)$, so the $\Lambda_{h,0}$-span of $U(\rho_0)$ contains a $\Lambda_{h,0}$-lattice in $\I^\circ_0$.
\end{proof}

We show that Proposition \ref{lambdaspan_0} is true if we replace $\rho_0$ by $\rho\vert_H$. 
This will be a consequence of the description of the subnormal sugroups of $\GL_2(\I^\circ)$ presented in \cite{taz}, but 
we need a preliminary step because we cannot induce a $\Lambda_{h,0}$-module structure on the unipotent subgroups of $G$. 
For a subgroup $\cG\subset\GL_2(\I^\circ_0)$ define $\cG^p=\{g^p,\, g\in G\}$ and $\widetilde{\cG}=\cG^p\cap(1+p\Mat_2(\I^\circ_0))$. Let $\widetilde{\cG}^{\Lambda_{h,0}}$ be the subgroup of $\GL_2(\I^\circ)$ generated by the set $\{g^\lambda\colon g\in\widetilde{\cG}, \lambda\in\Lambda_{h,0}\}$ where $g^\lambda=\exp(\lambda\log g)$. We have the following.

\begin{lemma}\label{uniplattice}
The group $\widetilde{\cG}^{\Lambda_{h,0}}$ contains a congruence subgroup of $\SL_2(\I^\circ_0)$ if and only if both of the unipotent subgroups $\cG\cap U^+(\I^\circ_0)$ and $\cG\cap U^-(\I^\circ_0)$ contain a basis of a $\Lambda_{h,0}$-lattice in $\I^\circ_0$.
\end{lemma}

\begin{proof}
It is easy to see that $\cG\cap U^+(\I^\circ_0)$ contains the basis of a $\Lambda_{h,0}$-lattice in $\I^\circ_0$ if and only if the same is true for $\widetilde{\cG}\cap U^+(\I^\circ_0)$. The same is true for $U^-$.
By a standard argument, used in the proofs of \cite[Lemma 2.9]{hida} and \cite[Prop. 4.2]{lang}, $\cG^{\Lambda_{h,0}}\subset\GL_2(\I^\circ_0)$ contains a congruence subgroup of $\SL_2(\I^\circ_0)$ if and only if both its upper and lower unipotent subgroup contain an ideal of $\I^\circ_0$.
We have $U^+(\I^\circ_0)\cap\cG^{\Lambda_{h,0}}=\Lambda_{h,0}(\cG\cap U^+(\I^\circ_0))$, so by Lemma \ref{lattice} $U^+(\I^\circ_0)\cap\cG^{\Lambda_{h,0}}$ contains an ideal of $\I^\circ_0$ if and only if $\cG\cap U^+(\I^\circ_0)$ contains a basis of a $\Lambda_{h,0}$-lattice in $\I^\circ_0$. We proceed in the same way for $U^-$.
\end{proof}

Now let $G_0=\im\,\rho\vert_H$, $G=\im\,\rho_0$. Note that $G_0\cap\SL_2(\I^\circ_0)$ is a normal subgroup of $G$.
Let $f\colon\GL_2(\I^\circ_0)\to\SL_2(\I^\circ_0)$ be the homomorphism sending $g$ to $\det(g)^{-1/2}g$.
We have $G=f(G_0)$ by definition of $\rho_0$. We show the following.

\begin{proposition}\label{unipbasis}
The subgroups $G_0\cap U^{\pm}(\I^\circ_0)$ both contain the basis of a $\Lambda_{h,0}$-lattice in $\I^\circ_0$ if and only if $G\cap U^{\pm}(\I^\circ_0)$ both contain the basis of a $\Lambda_{h,0}$-lattice in $\I^\circ_0$.
\end{proposition}

\begin{proof}
Since $G=f(G_0)$ we have $\widetilde{G}=f(\widetilde{G_0})$.
This implies that $\widetilde{G}^{\Lambda_{h,0}}=f(\widetilde{G_0}^{\Lambda_{h,0}})$.
We remark that $\widetilde{G_0}^{\Lambda_{h,0}}\cap\SL_2(\I_0^\circ)$ is a normal subgroup of $\widetilde{G}^{\Lambda_{h,0}}$.
Indeed $\widetilde{G_0}^{\Lambda_{h,0}}\cap\SL_2(\I_0^\circ)$ is normal in $\widetilde{G_0}^{\Lambda_{h,0}}$, so its image $f(G_0^{\Lambda_{h,0}}\cap\SL_2(\I_0^\circ))=G_0^{\Lambda_{h,0}}\cap\SL_2(\I_0^\circ)$ is normal in $f(G_0^{\Lambda_{h,0}})=\widetilde{G}^{\Lambda_{h,0}}$.

By \cite[Cor. 1]{taz} a subgroup of $\GL_2(\I^\circ_0)$ contains a congruence subgroup of $\SL_2(\I^\circ_0)$ if and only if it is subnormal in $\GL_2(\I^\circ_0)$ and it is not contained in the centre. We note that $\widetilde{G_0}^{\Lambda_{h,0}}\cap\SL_2(\I^\circ_0)=(\widetilde{G_0}\cap\SL_2(\I^\circ_0))^{\Lambda_{h,0}}$ is not contained in the subgroup $\{\pm 1\}$. Otherwise also $\widetilde{G_0}\cap\SL_2(\I^\circ_0)$ would be contained in $\{\pm 1\}$ and $\im\,\rho\cap\SL_2(\I^\circ_0)$ would be finite, since $\widetilde{G_0}$ is of finite index in $G_0^p$.
This would give a contradiction: indeed if $\fP$ is an arithmetic prime of $\I^\circ$ of weight greater than $1$ and $\fP^\prime=\fP\cap\I^\circ_0$, the image of $\rho$ modulo $\fP^\prime$ contains a congruence subgroup of $\SL_2(\I_0^\circ/\fP^\prime)$ by the result of \cite{ribet1}.

Since $\widetilde{G_0}^{\Lambda_{h,0}}\cap\SL_2(\I^\circ_0)$ is a normal subgroup of $\widetilde{G}^{\Lambda_{h,0}}$, we deduce by \cite[Cor. 1]{taz} that $\widetilde{G_0}^{\Lambda_{h,0}}\cap\SL_2(\I^\circ_0)$ (hence $\widetilde{G_0}^{\Lambda_{h,0}}$) contains a congruence subgroup of $\SL_2(\I^\circ_0)$ if and only if $\widetilde{G}^{\Lambda_{h,0}}$ does.
By applying Lemma \ref{uniplattice} to $\cG=G_0$ and $\cG=G$ we obtain the desired equivalence.
\end{proof}

By combining Propositions \ref{lambdaspan_0} and \ref{unipbasis} we obtain the following.

\begin{corollary}\label{lambdaspan}
The $\Lambda_{h,0}$-span of each of the unipotent subgroups $\im\,\rho\cap U^{\pm}$ contains a $\Lambda_{h,0}$-lattice in $\I^\circ_0$.
\end{corollary}

Unlike in the ordinary case we cannot deduce from the corollary that $\im\,\rho$ contains a congruence subgroup of $\SL_2(\I^\circ_0)$, since we are working over $\Lambda_h\neq\Lambda$ and we cannot induce a $\Lambda_h$-module structure (not even a $\Lambda$-module structure) on $\im\,\rho\cap U^{\pm}$.
The proofs of \cite[Lemma 2.9]{hida} and \cite[Prop. 4.3]{lang} rely on the existence, in the image of the Galois group, of an element inducing by conjugation a $\Lambda$-module structure on $\im\,\rho\cap U^{\pm}$. In their situation this is predicted by the condition of Galois ordinarity of $\rho$.
In the non-ordinary case we will find an element with a similar property via relative Sen theory.
In order to do this we will need to work with a suitably defined Lie algebra rather than with the group itself.

\bigskip

\section{Relative Sen theory}\label{sentheory}

We recall the notations of Section \ref{subsectnest}.
In particular $r_h=p^{-s_h}$, with $s_h\in\Q$, is the $h$-adapted radius (which we also take smaller than $p^{-\frac{1}{p-1}})$, $\eta_h$ is an element in $\C_p$ of norm $r_h$, $K_h$ is the Galois closure in $\C_p$ of $\Q_p(\eta_h)$ and $\cO_h$ is the ring of integers in $K_h$.
The ring $\Lambda_h$ of analytic functions bounded by $1$ on the open disc $\cB_h=\cB(0,r_h^-)$ is identified to $\cO_h[[t]]$. We take a sequence of radii $r_i=p^{-s_h-1/i}$ converging to $r_h$ and denote by $A_{r_i}=K_h\langle t,X_i\rangle/(pX_i-t^i)$ the $K_h$-algebra defined in Section \ref{subsectnest} which is a form over $K_h$ of the $\C_p$-algebra of analytic functions on the closed ball $\cB(0,r_i)$ (its Berthelot model). We denote by $A_{r_i}^\circ$ the $\cO_h$-subalgebra of functions bounded by $1$.
Then $\Lambda_h=\varprojlim_{i} A_{r_i}^\circ$ where $A_{r_j}^\circ\to A_{r_i}^\circ$ for $i<j$ is the restriction of analytic functions.

We defined in Section \ref{selftwists} a subring $\I_0^\circ\subset\I^\circ$, finite over $\Lambda_{h,0}\subset\Lambda_h$. For $r_i$ as above, we write $A_{0,r_i}^\circ=\cO_{h,0}\langle t,X_i\rangle/(pX_i-t^i)$ with maps $A_{0,r_j}^\circ\to A_{0,r_i}^\circ$ for $i<j$, so that $\Lambda_{h,0}=\varprojlim_iA_{0,r_i}^\circ$.
Let $\I_{r_i}^\circ=\I^\circ\widehat{\otimes}_{\Lambda_h} A_{r_i}^\circ$ and $\I_{0,r_i}^\circ=\I_0^\circ\widehat{\otimes}_{\Lambda_{h,0}} A_{0,r_i}^\circ$, both endowed with their $p$-adic topology. Note that $(\I_{r_i}^\circ)^\Gamma=\I_{r_i,0}^\circ$.

Consider the representation $\rho\colon G_\Q\to\GL_2(\I^\circ)$ associated with a family $\theta\colon \T_h\to\I^\circ$.
We observe that $\rho$ is continuous with respect to the profinite topology of $\I^\circ$ but not with respect to the $p$-adic topology.
For this reason we fix an arbitrary radius $r$ among the $r_i$ defined above and consider the representation $\rho_r\colon G_\Q\to\GL_2(\I_{r}^\circ)$ obtained by composing $\rho$ with the inclusion $\GL_2(\I^\circ)\into\GL_2(\I_{r}^\circ)$.
This inclusion is continuous, hence the representation $\rho_r$ is continuous with respect to the $p$-adic topology on $\GL_2(\I_{0,r}^\circ)$.

Recall from Proposition \ref{propI0} that, after replacing $\rho$ by a conjugate, there is an open normal subgroup $H_0\subset G_\Q$ such that the restriction $\rho\vert_{H_0}$ takes values in $\GL_2(\I_0^\circ)$ and is $(H_0,\Z_p)$-regular. Then the restriction $\rho_r\vert_{H_0}$ gives a representation $H_0\to\GL_2(\I_{0,r}^\circ)$ which is continuous with respect to the $p$-adic topology on $\GL_2(\I_{0,r}^\circ)$.


\subsection{Big Lie algebras}\label{liealg}

Recall that $G_p\subset G_\Q$ denotes our chosen decomposition group at $p$.
Let $G_r$ and $G_r^\loc$ be the images respectively of $H_0$ and $G_p\cap H_0$ under the representation $\rho_r\vert_{H_0}\colon H_0\to\GL_2(\I_{0,r}^\circ)$.
Note that they are actually independent of $r$ since they coincide with the images of $H_0$ and $G_p\cap H_0$ under $\rho$.

For every ring $R$ and ideal $I\subset R$ we denote by $\Gamma_{\GL_2(R)}(I)$ the $\GL_2$-congruence subgroup consisting of elements $g\in\GL_2(R)$ such that $g\equiv\Id_2\pmod{I}$.
Let $G_r^{\prime}=G_r\cap\Gamma_{\GL_2(\I_{0,r}^\circ)}(p)$ and $G_r^{\prime,\loc}=G_r^\loc\cap\Gamma_{\GL_2(\I_{0,r}^\circ)}(p)$, so that $G_r^{\prime}$ and $G_{r}^{\prime,\loc}$ are pro-$p$ groups.
Note that the congruence subgroups $\Gamma_{\GL_2(\I_{0,r})}(p^m)$ are open in $\GL_2(\I_{0,r})$ for the $p$-adic topology.
In particular $G_r^\prime$ and $G_{r}^{\prime,\loc}$ can be identified with the images under $\rho$ of the absolute Galois groups of  finite extensions  of $\Q$ and respectively $\Q_p$.

\begin{remark}\label{propsubgr}
We remark that we can choose an arbitrary $r_0$ and set, for every $r$, $G_r^{\prime}=G_r\cap\Gamma_{\GL_2(\I_{0,r_0}^\circ)}(p)$. 
Then $G_r^\prime$ is a pro-$p$ subgroup of $G_r$ for every $r$ and it is independent of $r$ since $G_r$ is. 
This will be important in Section \ref{comparison} where we will take projective limits over $r$ of various objects.
\end{remark}

We set $A_{0,r}=A_{0,r}^\circ[p^{-1}]$ and $\I_{0,r}=\I_{0,r}^\circ[p^{-1}]$.
We consider from now on $G_r^{\prime}$ and $G_r^{\prime,\loc}$ as subgroups of $\GL_2(\I_{0,r})$ through the inclusion $\GL_2(\I_{0,r}^\circ)\into\GL_2(\I_{0,r})$.


%
%
%
%

We want to define big Lie algebras associated with the groups $G_r^{\prime}$ and $G_r^{\prime,\loc}$. 
For every nonzero ideal $\fa$ of the principal ideal domain $A_{0,r}$, we denote by $G_{r,\fa}^\prime$ and $G_{r,\fa}^{\prime,\loc}$ the images respectively of $G_r^{\prime}$ and $G_r^{\prime,\loc}$ under the natural projection $\GL_2(\I_{0,r})\to\GL_2(\I_{0,r}/\fa \I_{0,r})$. The pro-$p$ groups $G_{r,\fa}^\prime$ and $G_{r,\fa}^{\prime,\loc}$ 
are topologically of finite type
so we can define the corresponding $\Q_p$-Lie algebras $\fH_{r,\fa}$ and $\fH_{r,\fa}^\loc$ using the $p$-adic logarithm map: $\fH_{r,\fa}=\Q_p\cdot\Log\, G_{r,\fa}^{\prime}$ and $\fH_{r,\fa}^\loc=\Q_p\cdot\Log\, G_{r,\fa}^{\prime,\loc}$.
They are closed Lie subalgebras of the finite dimensional $\Q_p$-Lie algebra $\Mat_2(\I_{0,r}/\fa\I_{0,r})$.



Let $B_r=\varprojlim_{(\fa,P_1)=1}A_{0,r}/ \fa A_{0,r}$ where the inverse limit is taken over nonzero ideals $\fa\subset A_{0,r}$ prime to 
$P_1=(u^{-1}(1+T)-1)$ (the reason for excluding $P_1$ will become clear later).
We endow $B_r$ with the projective limit topology coming from the $p$-adic topology on each quotient. 
We have a topological isomorphism of $K_{h,0}$-algebras
$$ B_r\cong\prod_{P\neq P_1} \widehat{(A_{0,r})}_{P}, $$
where the product is over primes $P$  and $\widehat{(A_{0,r})}_P=\varprojlim_{m\geq 1}A_{0,r}/P^mA_{0,r}$ 
denotes the $K_{h,0}$-Fr\'echet space inverse limit of the finite dimensional $K_{h,0}$-vector spaces $A_{0,r}/P^mA_{0,r}$.
Similarly, let $\B_r=\varprojlim_{(\fa,P_1)=1}\I_{0,r}/\fa\I_{0,r}$, where as before $\fa$ varies over 
all nonzero ideals of $A_{0,r}$ prime to $P_1$. We have
$$ \B_r\cong\prod_{P\neq P_1} \widehat{(\I_{0,r})}_{P\I_{0,r}}\cong\prod_{\fP\nmid P_1} \widehat{(\I_{0,r})}_{\fP}\cong\varprojlim_{(\fQ,P_1)=1}\I_{0,r}/\fQ, $$
where the second product is over primes $\fP$ of $\I_{0,r}$ and the projective limit is over primary ideals $\fQ$ of $\I_{0,r}$. Here $\widehat{(\I_{0,r})}_\fP$ denotes the projective limit of finite dimensional $K_{h,0}$-algebras (endowed with the $p$-adic topology).
The last isomorphism follows from the fact that $\I_{0,r}$ is finite over $A_{0,r}$, so that there is an isomorphism $\I_{0,r}\otimes\widehat{(A_{0,r})}_P=\prod_\fP\widehat{(\I_{0,r})}_\fP$ where $P$ is a prime of $A_{0,r}$ and $\fP$ 
varies among the primes of $\I_{0,r}$ above $P$.
We have natural continuous inclusions $A_{0,r}\into B_r$ and $\I_{0,r}\into \B_r$, both with dense image. The map $A_{0,r}\into\I_{0,r}$ induces an inclusion $B_r\into\B_r$ with closed image. Note however that $\B_r$ is not finite over $B_r$.
We will work with $\B_r$ for the rest of this section, but we will need $B_r$ later.

For every $\fa$ we have defined Lie algebras $\fH_{r,\fa}$ and $\fH_{r,\fa}^\loc$ associated with the finite type Lie groups $G_{r,\fa}^\prime$ and $G_{r,\fa}^{\prime,\loc}$. We take the projective limit of these algebras to obtain Lie subalgebras of $\Mat_2(\B_r)$.

\begin{definition}\label{blalg}
The Lie algebras associated with $G_{r}^\prime$ and $G_{r}^{\prime,\loc}$ are the closed $\Q_p$-Lie subalgebras of $\Mat_2(\B_r)$ given respectively by
$$ \fH_{r}=\varprojlim_{(\fa,P_1)=1}\fH_{r,\fa} $$
and
$$ \fH_{r}^\loc=\varprojlim_{(\fa,P_1)=1}\fH_{r,\fa}^\loc, $$
where as usual the products are taken over nonzero ideals $\fa\subset A_{0,r}$ prime to $P_1$.
\end{definition}


For every ideal $\fa$ prime to $P_1$, we have continuous homomorphisms $\fH_r\to\fH_{r,\fa}$ and $\fH_r^\loc\to\fH_{r,\fa}^\loc$.
Since the transition maps are surjective these homomorphisms are surjective.

\begin{remark}\label{primlim}
The limits in Definition \ref{blalg} can be replaced by limits over primary ideals of $\I_{0,r}$. Explicitly, let $\fQ$ be a primary ideal of $\I_{0,r}$. Let $G_{r,\fQ}^\prime$ be the image of $G_r^\prime$ via the natural projection $\GL_2(\I_{0,r})\to\GL_2(\I_{0,r}/\fQ)$ and let $\fH_{r,\fQ}$ be the Lie algebra associated with $G_{r,\fQ}^\prime$ (which is a finite type Lie group). We have an isomorphism of topological Lie algebras
$$ \fH_{r}=\varprojlim_{(\fQ,P_1)=1}\fH_{r,\fQ}, $$
where the limit is taken over primary ideals $\fQ$ of $\I_{0,r}$. This is naturally a subalgebra of $\Mat_2(\B_r)$ since $\B_r\cong\varprojlim_{(\fQ,P_1)=1}\I_{0,r}/\fQ$. The same goes for the local algebras.
\end{remark}

\subsection{The Sen operator associated with a Galois representation}\label{senconstr}


Recall that there is a finite extension $K/\Q_p$ such that $G_r^{\prime,\loc}$ is the image of $\rho\vert_{\Gal(\overline{K}/K)}$ 
and, for an ideal $P\subset A_{0,r}$ and $m\geq 1$, $G_{r,P^m}^{\prime,\loc}$ is the image of $\rho_{r,P^m}\vert_{\Gal(\overline{K}/K)}$. Following \cite{sen1} and \cite{sen2} we can define a Sen operator 
associated with $\rho_r\vert_{\Gal(\overline{K}/K)}$ and $\rho_{r,P^m}\vert_{\Gal(\overline{K}/K)}$ 
for every ideal $P\subset A_{0,r}$ and every $m\geq 1$. We will see that these operators satisfy a compatibility property. 
We write for the rest of the section $\rho_r$ and $\rho_{r,P^m}$ while implicitly taking the domain to be $\Gal(\overline{K}/K)$.

We begin by recalling the definition of the Sen operator associated with a representation $\tau\colon\Gal(\overline{K}/K)\to\GL_m(\cR)$ 
where $\cR$ is a Banach algebra over a $p$-adic field $L$. We follow \cite{sen2}. We can suppose $L\subset K$; 
if not we just restrict $\tau$ to the open subgroup $\Gal(\overline{K}/KL)\subset\Gal(\overline{K}/K)$.

Let $L_{\infty}$ be a totally ramified $\Z_p$-extension of $L$.
Let $\gamma$ be a topological generator of $\Gamma=\Gal(L_{\infty}/L)$, $\Gamma_n\subset\Gamma$ the subgroup generated by $\gamma^{p^n}$ and $L_{n}=L_{\infty}^{\gamma^{p^n}}$, so that $L_{\infty}=\cup_n L_{n}$.
Let $L_n^\prime=L_{n}K$ and $G_n^\prime=\Gal(\overline{L}/L_n^\prime)$.
If $\cR^m$ is the $\cR$-module over which $\Gal(\overline{K}/K)$ acts via $\tau$, define an action of $\Gal(\overline{K}/K)$ on $\cR\widehat{\otimes}_L \C_p$ by letting $\sigma\in\Gal(\overline{K}/K)$ map $x\otimes y$ to $\tau(\sigma)(x)\otimes\sigma(y)$.
Then by the results of \cite{sen1} and \cite{sen2} there is a matrix
$M\in \GL_m\left(\cR\widehat{\otimes}_L \C_p\right)$, an integer $n\ge 0$ and a representation
$\delta\colon\Gamma_n\to \GL_m(\cR\otimes_L L_{n}^\prime)$ such that for all $\sigma\in G_n^\prime$
$$ M^{-1}\tau(\sigma)\sigma(M)=\delta(\sigma). $$

\begin{definition}
The Sen operator associated with $\tau$ is
$$ \phi=\lim_{\sigma\to 1}\frac{\log(\delta\bigl(\sigma)\bigr)}{\log(\chi(\sigma))}\in \Mat_m(\cR\widehat{\otimes}_L \C_p). $$
\end{definition}

The limit exists as for $\sigma$ close to $1$ the map $\displaystyle \sigma\mapsto 
\frac{\log(\delta\bigl(\sigma)\bigr)}{\log(\chi(\sigma))}$ is constant.
It is proved in \cite[Sec. 2.4]{sen2} that $\phi$ does not depend on the choice of $\delta$ and $M$.

If $L=\cR=\Q_p$, we define the Lie algebra $\fg$ associated with $\tau(\Gal(\overline{K}/K))$ as the $\Q_p$-vector space generated by the image of the $\Log$ map in $\Mat_m(\Q_p)$. In this situation the Sen operator $\phi$ associated with $\tau$ has the following property.

\begin{theorem}\label{liealgsen}\cite[Th. 1]{sen1}
For a continuous representation $\tau\colon G_K\to\GL_m(\Q_p)$, the Lie algebra $\fg$ of the group $\tau(\Gal(\overline{K}/K))$ is the smallest $\Q_p$-subspace of $\Mat_m(\Q_p)$ such that $\fg{\otimes} \C_p$ contains $\phi$.
\end{theorem}

\noindent This theorem is valid in the absolute case above, but relies heavily on the fact that the image of the Galois group is a finite dimensional Lie group.
In the relative case it is doubtful that its proof can be generalized.

\subsection{The Sen operator associated with $\rho_r$}

Set $\I_{0,r,\C_p}=\I_{0,r}\widehat{\otimes}_{K_{h,0}}\C_p$. It is a Banach space for the natural norm.
Let $\B_{r,\C_p}=\B_r\widehat{\otimes}_{K_{h,0}}\C_p$; it is the topological $\C_p$-algebra completion of $\B_r\otimes_{K_{h,0}} \C_p$ for the (uncountable) 
set of nuclear seminorms $p_{\fa}$ given by the norms on $\I_{0,r,\C_p}/\fa\I_{0,r,\C_p}$ via the specialization morphisms
$\pi_\fa\colon \B_r\otimes_{K_{h,0}}\C_p\to \I_{0,r,\C_p}/\fa\I_{0,r,\C_p}$.
Let $\fH_{r,\fa,\C_p}=\fH_{r,\fa}\otimes_{K_{h,0}}\C_p$ and $\fH_{r,\fa,\C_p}^\loc=\fH_{r,\fa,}^\loc\otimes_{K_{h,0}}\C_p$.
Then we define $\fH_{r,\C_p}=\fH_r\widehat{\otimes}_{K_{h,0}}\C_p$ as the topological $\C_p$-Lie algebra completion of 
$\fH_r\otimes_{K_{0,h}}\C_p$ for the (uncountable) 
set of seminorms $p_\fa$ given by the norms on $\fH_{r,\fa,\C_p}$ and similar specialization morphisms $\pi_\fa\colon\fH_{r,}\otimes_{K_{h,0}}\C_p\to\fH_{r,\fa,\C_p}$. We define in the same way $\fH_{r,\C_p}^\loc$ in terms of the norms on $\fH^{\loc}_{r,\fa,\C_p}$. Note that by definition we have 
$$\fH_{r,\C_p}=\varprojlim_{(\fa,P_1)=1}\fH_{r,\fa,\C_p},\,\,\mathrm{and}\,\, \fH_{r,\C_p}^\loc=\varprojlim_{(\fa,P_1)=1}\fH_{r,\fa,\C_p}^\loc.$$

We apply the construction of the previous subsection to $L=K_{h,0}$, $\cR=\I_{0,r}$ which is a Banach $L$-algebra with the $p$-adic topology, and $\tau=\rho_r$. We obtain an operator $\phi_r\in \Mat_2(\I_{0,r,\C_p})$.
Recall that we have a natural continuous inclusion $\I_{0,r}\into \B_r$, inducing inclusions $\I_{0,r,\C_p}\into \B_{r,\C_p}$ and $\Mat_2(\I_{0,r,\C_p})\into \Mat_2(\B_{r,\C_p})$. We denote all these inclusions by $\iota_{\B_r}$ since it will be clear each time to which we are referring to. 
We will prove in this section that $\iota_{\B_r}(\phi_r$) is an element of $\fH_{r,\C_p}^\loc$.


Let $\fa$ be a nonzero ideal of $A_{0,r}$. Let us apply Sen's construction to $L=K_{h,0}$, $\cR=\I_{0,r}/\fa\I_{0,r}$ and $\tau=\rho_{r,\fa}\colon\Gal(\overline{K}/K)\to\GL_2(\I_{0,r}/\fa\I_{0,r})$; we obtain an operator $\phi_{r,\fa}\in \Mat_2(\I_{0,r}/\fa\I_{0,r}\widehat{\otimes}_{K_{h,0}}\C_p)$.

Let 
$$\pi_\fa\colon\Mat_2(\I_{0,r}\widehat{\otimes}_{K_{h,0}}\C_p)\to \Mat_2(\I_{0,r}/\fa\I_{0,r}\widehat{\otimes}_{K_{h,0}}\C_p)$$
and 
$$\pi_\fa^\times\colon\GL_2(\I_{0,r}\widehat{\otimes}_{K_{h,0}}\C_p)\to\GL_2(\I_{0,r}/\fa\I_{0,r}\widehat{\otimes}_{K_{h,0}}\C_p)$$
be the natural projections.

\begin{proposition}\label{senproj}
We have $\phi_{r,\fa}=\pi_\fa(\phi_r)$ for all $\fa$.
\end{proposition}

\begin{proof}
Recall from the construction of $\phi_r$ that there exist $M\in\GL_2\left(\I_{0,r,\C_p}\right)$, $n\geq 0$ and
$\delta\colon\Gamma_n\to\GL_2(\I_{0,r}\widehat{\otimes}_{K_{h,0}}K_{h,0,n}^\prime)$ such that for all $\sigma\in G_n^\prime$ we have
\begin{equation}\label{sigmadelta}
M^{-1}\rho_r(\sigma)\sigma(M)=\delta(\sigma)
\end{equation}
and
\begin{equation}\label{eqsen}
\phi_r=\lim_{\sigma\to 1}\frac{\log(\delta\bigl(\sigma)\bigr)}{\log(\chi(\sigma))}.
\end{equation}

Let $M_\fa=\pi^\times_\fa(M)\in\GL_2(\I_{0,r,\C_p}/\fa\I_{0,r,\C_p})$ and $\delta_\fa=\pi_\fa^\times\circ\delta\colon\Gamma_n\to\GL_2((\I_{0,r}/\fa\I_{0,r})\widehat{\otimes}_{K_{h,0}}K_{h,0,n}^\prime)$
Denote by $\phi_{r,\fa}\in \Mat_2((\I_{0,r}/\fa\I_{0,r})\widehat{\otimes}_{K_{h,0}}K_{h,0,n}^\prime)$ the Sen operator associated with $\rho_{r,\fa}$.
Now (\ref{sigmadelta}) gives
\begin{equation}\label{sigmadeltaQ}
M_\fa^{-1}\rho_{r,\fa}(\sigma)\sigma(M_\fa)=\delta_\fa(\sigma)
\end{equation}
so we can calculate $\phi_{r,\fa}$ as
\begin{equation}\label{eqsenQ}
\phi_{r,\fa}=\lim_{\sigma\to 1}\frac{\log(\delta_\fa\bigl(\sigma)\bigr)}{\log(\chi(\sigma))}\in \Mat_2(\cR\widehat{\otimes}_L \C_p).
\end{equation}

\noindent By comparing this with (\ref{eqsen}) we see that $\phi_{r,\fa}=\pi_\fa(\phi_r)$.
\end{proof}

Let $\phi_{r,\B_r}=\iota_{\B_r}(\phi_r)$. For a nonzero ideal $\fa$ of $A_{0,r}$ let $\pi_{\B_r,\fa}$ be the natural projection $\B_r\to\I_{0,r}/\fa\I_{0,r}$. Clearly $\pi_{\B_r,\fa}(\phi_{r,\B_r})=\pi_{\fa}(\phi_r)$ and $\phi_{r,\fa}=\pi_{\fa}(\phi_r)$ by Proposition \ref{senproj}, so we have $\phi_{r,\B_r}=\varprojlim_{(\fa,P_1)=1}\phi_{r,\fa}$.

We apply Theorem \ref{liealgsen} to show the following.

\begin{proposition}\label{liealgsenQ}
Let $\fa$ be a nonzero ideal of $A_{0,r}$ prime to $P_1$. The operator $\phi_{r,\fa}$ belongs to the Lie algebra $\fH_{r,\fa,\C_p}^\loc$.
\end{proposition}

\begin{proof}
Let $n$ be the dimension over $\Q_p$ of $\I_{0,r}/\fa\I_{0,r}$; by choosing a $\Q_p$-basis $(\omega_1,\ldots,\omega_n)$ of this algebra,
we can define an injective ring morphism $\alpha\colon\Mat_2(\I_{0,r}/\fa\I_{0,r})\into \Mat_{2n}(\Q_p)$ and an injective group morphism $\alpha^\times\colon\GL_2(\I_{0,r}/Q\I_{0,r})\into\GL_{2n}(\Q_p)$. In fact, an endomorphism $f$ of the $(\I_{0,r}/\fa\I_{0,r})$-module 
$(\I_{0,r}/\fa\I_{0,r})^2=(\I_{0,r}/\fa\I_{0,r})\cdot e_1\oplus (\I_{0,r}/\fa\I_{0,r})\cdot e_2$ is $\Q_p$-linear, so it induces an endomorphism $\alpha(f)$ of the $\Q_p$-vector space $(\I_{0,r}/\fa\I_{0,r})^2=\bigoplus_{i,j} \Q_p\cdot \omega_i e_j$; furthermore if $\alpha$ is an automorphism then $\alpha(f)$ is one too.
In particular $\rho_{r,\fa}$ induces a representation $\rho_{r,\fa}^\alpha=\alpha^\times\circ\rho_{r,\fa}\colon\Gal(\overline{K}/K)\to\GL_{2n}(\Q_p)$. The image of $\rho_{r,\fa}^\alpha$ is the group $G_{r,\fa}^{\loc,\alpha}=\alpha^\times(G_{r,\fa}^\loc)$.
We consider its Lie algebra $\fH_{r,\fa}^{\loc,\alpha}=\Q_p\cdot\Log\,(G_{r,\fa}^{\loc,\alpha})\subset \Mat_{2n}(\Q_p)$. 
The $p$-adic logarithm commutes with $\alpha$ in the sense that $\alpha(\Log\, x)=\Log\,(\alpha^\times(x))$ for every $x\in\Gamma_{\I_{0,r}/\fa\I_{0,r}}(p)$, so we have $\fH_{r,\fa}^{\loc,\alpha}=\alpha(\fH_{r,\fa}^\loc)$ (recall that $\fH_{r,\fa}^\loc=\Q_p\cdot \Log\, G_{r,\fa}^\loc)$.

Let $\phi_{r,\fa}^\alpha$ be the Sen operator associated with $\rho_{r,\fa}^\alpha\colon\Gal(\overline{K}/K)\to\GL_{2n}(\Q_p)$. By Theorem \ref{liealgsen} we have $\phi_{r,\fa}^\alpha\in\fH_{r,\fa,\C_p}^{\loc,\alpha}=\fH_{r,\fa}^{\loc,\alpha}\widehat{\otimes}\C_p$.
Denote by $\alpha_{\C_p}$ the map $\alpha\widehat{\otimes}1\colon\Mat_2(\I_{0,r,\C_p}/\fa\I_{0,r,\C_p})\into \Mat_{2n}(\C_p)$.
We show that $\phi_{r,\fa}^{\alpha_{\C_p}}=\alpha_{\C_p}(\phi_{r,\fa})$, from which it follows that $\phi_{r,\fa}\in\fH_{r,\fa,\C_p}^\loc$ since $\fH_{r,\fa,\C_p}^{\loc,\alpha_{\C_p}}=\alpha_{\C_p}(\fH_{r,\fa,{\C_p}}^\loc)$ and $\alpha_{\C_p}$ is injective. Now let $M_\fa$, $\delta_\fa$ be as in (\ref{sigmadeltaQ}) and $M_\fa^{\alpha_{\C_p}}=\alpha_{\C_p}(M_\fa)$, $\delta_\fa^{\alpha_{\C_p}}=\alpha_{\C_p}\circ\delta_\fa$.
By applying $\alpha_C$ to (\ref{sigmadelta}) we obtain $(M_\fa^{\alpha_{\C_p}})^{-1}\rho_{r,\fa}^{\alpha_{\C_p}}(\sigma)\sigma(M_\fa^{\alpha_{\C_p}})=\delta_\fa^{\alpha_{\C_p}}(\sigma)$ for every $\sigma\in G_n^\prime$, so we can calculate
$$ \phi_{r,\fa}^{\alpha_{\C_p}}=\lim_{\sigma\to 1}\frac{\log(\delta_\fa^{\alpha_{\C_p}}\bigl(\sigma)\bigr)}{\log(\chi(\sigma))}, $$
which coincides with $\alpha_{\C_p}(\phi_{r,\fa})$.
\end{proof}


\begin{proposition}\label{seninalg}
The element $\phi_{r,\B_r}$ belongs to $\fH_{r,{\C_p}}^\loc$, hence to $\fH_{r,{\C_p}}$.
\end{proposition}

\begin{proof}
By definition of the space $\fH_{r,{\C_p}}^\loc$ as completion 
of the space $\fH_{r}^\loc\otimes_{K_{h,0}}\C_p$ for the seminorms $p_\fa$ 
given by the norms on $\fH_{r,\fa,{\C_p}}^\loc$, we have
$\fH_{r,{\C_p}}^\loc=\varprojlim_{(\fa,P_1)=1}\fH_{r,\fa,{\C_p}}^\loc$.
By Proposition \ref{senproj}, we have $\phi_{r,\B_r}=\varprojlim_{\fa}\phi_{r,\fa}$ and by Proposition \ref{liealgsenQ} we have for every $\fa$, $\phi_{r,\fa}\in\fH_{r,\fa,{\C_p}}$. We conclude that $\phi_{r,\B_r}\in\fH_{r,{\C_p}}^\loc$.
\end{proof}


\begin{remark}\label{primlimsen}
In order to prove that our Lie algebras are ``big'' it will be useful to work with primary ideals of $A_r$, as we did in this subsection. However, in light of Remark \ref{primlim}, all of the results can be rewritten in terms of primary ideals $\fQ$ of $\I_{0,r}$. This will be useful in the next subsection, when we will interpolate the Sen operators corresponding to the classical modular representations.
\end{remark}

From now on we identify $\I_{0,r,{\C_p}}$ with a subring of $\B_{r,{\C_p}}$ via $\iota_{\B_r}$, so we also identify $\Mat_2(\I_{0,r})$ with a subring of $\Mat_2(\B_r)$ and $\GL_2(\I_{0,r,{\C_p}})$ with a subgroup of $\GL_2(\B_{r,{\C_p}})$. In particular we identify $\phi_r$ with $\phi_{r,\B_r}$ and we consider $\phi_r$ as an element of $\fH_{r,{\C_p}}\cap\Mat_2(\I_{0,r,{\C_p}})$.

\subsection{The characteristic polynomial of the Sen operator}\label{charpolyn}

Sen proved the following result.

\begin{theorem}\label{charpolsen}
Let $L_1$ and $ L_2$ be two $p$-adic fields. 
Assume for simplicity that $L_2$ contains the normal closure of $L_1$. Let
$\tau\colon\Gal(\overline{L}_1/L_1)\to\GL_m(L_2)$ be a continuous representation.
For each embedding $\sigma\colon L_1\to L_2$,
there is a Sen operator $\phi_{\tau,\sigma}\in \Mat_m(\C_p\otimes_{L_1,\sigma}L_2)$ associated with $\tau$ and $\sigma$.
If $\tau$ is Hodge-Tate and its Hodge-Tate weights with respect to $\sigma$ are $h_{1,\sigma},\ldots,h_{m,\sigma}$ (with multiplicities, if any), then the characteristic polynomial of $\phi_{\tau,\sigma}$ is $\prod_{i=1}^m(X-h_{i,\sigma})$. 
\end{theorem}

Now let $k\in\N$ and $P_k=(u^{-k}(1+T)-1)$ be the corresponding arithmetic prime of $A_{0,r}$.
Let $\fP_f$ be a prime of $\I_r$ above $P$, associated with the system of Hecke eigenvalues of a classical modular form $f$.
Let $\rho\colon\G_Q\to\GL_2(\I_r)$ be as usual. The specialization of $\rho$ modulo $\fP$ is the representation $\rho_f\colon G_\Q\to\GL_2(\I_r/\fP)$ classically associated with $f$, defined over the field $K_f=\I_r/\fP\I_r$. By a theorem of Faltings (\cite{faltings}), when the weight of the form $f$ is $k$, the representation $\rho_f$ is Hodge-Tate of Hodge-Tate weights $0$ and $k-1$. Hence by Theorem \ref{charpolsen} the Sen operator $\phi_f$ associated with $\rho_f$ has characteristic polynomial $X(X-(k-1))$.
Let $\fP_{f,0}=\fP_f\cap\I_{0,r}$. With the notations of the previous subsection, the specialization of $\rho_r$ modulo $\fP_{f,0}$ gives a representation $\rho_{r,\fP_{f,0}}\colon\Gal(\overline{K}/K)\to\GL_2(\I_{0,r}/\fP_{f,0})$, which coincides with $\rho_f\vert_{\Gal(\overline{K}/K)}$. In particular the Sen operator $\phi_{r,\fP_{f,0}}$ associated with $\rho_{r,\fP_{f,0}}$ is $\phi_f$.

By Proposition \ref{senproj} and Remark \ref{primlimsen}, the Sen operator $\phi_r\in \Mat_2(\I_{0,r,\C_p})$ specializes modulo $\fP_{f,0}$ to the Sen operator $\phi_{r,\fP_{f,0}}$ associated with $\rho_{r,\fP_{f,0}}$, for every $f$ as above. 
Since the primes of the form $\fP_{f,0}$ are dense in $\I_{0,r,\C_p}$, the eigenvalues of $\phi_{r,Q}$ are given by the unique interpolation of those of $\rho_{r,\fP_{f,0}}$. 
This way we will recover an element of $\GL_2(\B_{r,\C_p})$ with the properties we need.

Given $f\in A_{0,r}$ we define its $p$-adic valuation by $v_p^\prime(f)=\inf_{x\in\cB(0,r)}v_p(f(x))$, where $v_p$ is our chosen valuation on $\C_p$.
Then if $v^\prime(f-1)\leq p^{-\frac{1}{p-1}}$ there are well-defined elements $\log(f)$ and $\exp(\log(f))$ in $A_{0,r}$, and $\exp(\log(f))=f$. 

Let $\phi^\prime_r=\log(u)\phi_r$. 
Note that $\phi^\prime_r$ is a well-defined element of $\Mat_2(\B_{r,\C_p})$ since $\log(u)\in\Q_p$.
Recall that we denote by $C_T$ the matrix $\diag(u^{-1}(1+T),1)$.
We have the following.

\begin{proposition}\label{sendiag}
\begin{enumerate}[leftmargin=*]
\item The eigenvalues of $\phi^\prime_r$ are $\log(u^{-1}(1+T))$ and $0$.
In particular the exponential $\Phi_r=\exp(\phi^\prime_r)$ is defined in $\GL_2(\B_{r,\C_p})$.
Moreover $\Phi^\prime_r$ is conjugate to $C_T$ in $\GL_2(\B_{r,\C_p})$.
\item The element $\Phi^\prime_r$ of part (1) normalizes $\fH_{r,{\C_p}}$.
\end{enumerate}
\end{proposition}

\begin{proof}
For every $\fP_{f,0}$ as in the discussion above, the element $\log(u)\phi_r$ specializes to $\log(u)\phi_{r,\fP_{f,0}}$ modulo $\fP_{f,0}$. If $\fP_{f,0}$ is a divisor of $P_k$, the eigenvalues of $\log(u)\phi_{r,\fP_{f,0}}$ are $\log(u)(k-1)$ and $0$. 
Since $1+T=u^k$ modulo $\fP_{f,0}$ for every prime $\fP_{f,0}$ dividing $P_k$, we have $\log(u^{-1}(1+T))=\log(u^{k-1})=(k-1)\log(u)$ modulo $\fP_{f,0}$. Hence the eigenvalues of $\log(u)\phi_{r,\fP_{f,0}}$ are interpolated by $\log(u^{-1}(1+T))$ and $0$.

Recall that in Section \ref{subsectnest} we chose $r_h$ smaller than $p^{-\frac{1}{p-1}}$.
Since $r<r_h$, $v_p^\prime(T)<p^{-\frac{1}{p-1}}$. In particular $\log(u^{-1}(1+T))$ is defined and $\exp(\log(u^{-1}(1+T)))=u^{-1}(1+T)$, so $\Phi_r=\exp(\phi^\prime_r)$ is also defined and its eigenvalues are $u^{-1}(1+T)$ and $1$.
The difference between the two is $u^{-1}(1+T)-1$; this belongs to $P_1$, hence it is invertible because of our definition of $\B_r$. This proves (1).

By Proposition \ref{seninalg}, $\phi_r\in\fH_{r,{\C_p}}$. Since $\fH_{r,{\C_p}}$ is a $\Q_p$-Lie algebra, $\log(u)\phi_r$ is also an element of $\fH_{r,{\C_p}}$. Hence its exponential $\Phi^\prime_r$ normalizes $\fH_{r,{\C_p}}$.
\end{proof}

\bigskip

\section{Existence of the Galois level for a family with finite positive slope}

Let $r_h\in p^\Q\cap]0,p^{-\frac{1}{p-1}}[$ be the radius chosen in Section \ref{sectcong}. As usual we write $r$ for any one of the radii $r_i$ of Section \ref{subsectnest}.
Recall that $\fH_{r}\subset \Mat_2(\B_r)$ is the Lie algebra attached to the image of $\rho_{r}$ (see Definition \ref{blalg}) 
and $\fH_{0,r,\C_p}=\fH_{r}\widehat{\otimes}\C_p$.
Let $\fu^\pm$, respectively $\fu^\pm_{\C_p}$, be the upper and lower nilpotent subalgebras of $\fH_{r}$, respectively $\fH_{r,\C_p}$.

\begin{remark} \label{remind_r} The commutative Lie algebra $\fu^\pm$ is independent of $r$ because it is equal to $\Q_p\cdot\Log(U(\I_0^\circ)\cap G_r^\prime)$ which is independent of $r$, provided $r_0\leq r<r_h$.
\end{remark}

We fix $r_0\in p^\Q\cap]0,r_h[$ arbitrarily and we work from now on with radii $r$ satisfying $r_0\leq r<r_h$. 
As in Remark \ref{propsubgr} this fixes a finite extension of $\Q$ corresponding to the inclusion $G_r^\prime\subset G_r$.
For $r<r^\prime$ we have a natural inclusion $\I_{0,r^\prime}\into\I_{0,r}$. 
Since $\B_r=\varprojlim_{(\fa P_1)=1}\I_{0,r}/\fa \I_{0,r}$ this induces an inclusion $\B_{r^\prime}\into\B_r$.
We will consider from now on $\B_{r^\prime}$ as a subring of $\B_r$ for every $r<r^\prime$. We will also consider $\Mat_2(\I_{0,r^\prime,\C_p})$ and $\Mat_2(\B_{r^\prime})$ as subsets of $\Mat_2(\I_{0,r,\C_p})$ and $\Mat_2(\B_r)$ respectively.
These inclusions still hold after taking completed tensors with $\C_p$.




Recall the elements $\phi^\prime_r=\log(u)\phi_r\in \Mat_2(\B_{r,\C_p})$ and $\Phi^\prime_r=\exp(\phi^\prime_r)\in\GL_2(\I_{0,r,\C_p})$ defined at the end of the previous section.
The Sen operator $\phi_r$ is independent of $r$ in the following sense: if $r<r^\prime<r_h$ and $\B_{r^\prime,\C_p}\to\B_{r,\C_p}$ is the natural inclusion then the image of $\phi_{r^\prime}$ under the induced map $\Mat_2(\B_{r^\prime,\C_p})\to \Mat_2(\B_{r,\C_p})$ is $\phi_r$.
We deduce that $\phi^\prime_r$ and $\Phi^\prime_r$ are also independent of $r$ (in the same sense).

By Proposition \ref{sendiag}, for every $r<r_h$
there exists an element $\beta_r\in\GL_2(\B_{r,\C_p})$ such that $\beta_r\Phi^\prime_r\beta_r^{-1}=C_T$. Since $\Phi^\prime_r$ normalizes $\fH_{r,\C_p}$,  $C_T=\beta_r\Phi^\prime_r\beta_r^{-1}$ normalizes $\beta_r\fH_{r,\C_p}\beta_r^{-1}$ .





We denote by $\fU^{\pm}$ the upper and lower nilpotent subalgebras of $\fsl_2$. The action of $C_T$ on $\fH_{r,\C_p}$ by conjugation is semisimple, so we can decompose $\beta_r\fH_{r,\C_p}\beta_r^{-1}$ as a sum of eigenspaces for $C_T$:
$$ \beta_r\fH_{r,\C_p}\beta_r^{-1}=\left(\beta_r\fH_{r,\C_p}\beta_r^{-1}\right)[1]\oplus \left(\beta_r\fH_{r,\C_p}\beta_r^{-1}\right)[u^{-1}(1+T)]\oplus \left(\beta_r\fH_{r,\C_p}\beta_r^{-1}\right)[u(1+T)^{-1}] $$
with
$$ \left(\beta_r\fH_{r,\C_p}\beta_r^{-1}\right)[u^{-1}(1+T)]\subset \fU^+(\B_{r,\C_p}) \mbox{\qquad and \qquad} \left(\beta_r\fH_{r,\C_p}\beta_r^{-1}\right)[u(1+T)^{-1}]\subset\fU^-(\B_{r,\C_p}). $$

\noindent Moreover, the formula $$\left(\begin{array}{cc}u^{-1}(1+T)&0\\0&1\end{array}\right)
\left(\begin{array}{cc}1&\lambda\\0&1\end{array}\right)
\left(\begin{array}{cc}u^{-1}(1+T)&0\\0&1\end{array}\right)^{-1}=
\left(\begin{array}{cc}1&u^{-1}(1+T)\lambda\\0&1\end{array}\right)$$
shows that the action of $C_T$ by conjugation coincides with multiplication by $u^{-1}(1+T)$. 
By linearity this gives an action of the polynomial ring $\C_p[T]$ on $\beta_r\fH_{r,\C_p}\beta_r^{-1}\cap\fU^+(\B_{r,\C_p})$, 
compatible with the action of ${\C_p}[T]$ on $\fU^+(\B_{r,\C_p})$ given by the inclusions 
$\C_p[T]\subset\Lambda_{h,0,{\C_p}}\subset B_{r,\C_p}\subset\B_{r,\C_p}$. Since ${\C_p}[T]$ is dense in $A_{h,0,{\C_p}}$
for the $p$-adic topology, it is also dense in $B_{r,\C_p}$. Since $\fH_{r,\C_p}$ is a closed Lie subalgebra of $\Mat_2(\B_{r,\C_p})$, 
we can define by continuity a $B_{r,\C_p}$-module structure on 
$\beta_r\fH_{r,\C_p}\beta_r^{-1}\cap\fU^+(\B_{r,\C_p})$, compatible with that on $\fU^+(\B_{r,\C_p})$. Similarly
we have
$$ \left(\begin{array}{cc}u^{-1}(1+T)&0\\0&1\end{array}\right)\left(\begin{array}{cc}1&0\\ \mu &1\end{array}\right)\left(\begin{array}{cc}u^{-1}(1+T)&0\\0&1\end{array}\right)^{-1}=
\left(\begin{array}{cc}1&0\\u(1+T)^{-1}\mu &1\end{array}\right). $$

We note that $1+T$ is invertible in $A_{0,r}$ since $T=p^{s_h}t$ where $r_h=p^{-s_h}$. Therefore $C_T$ is invertible and by twisting by $(1+T)\mapsto (1+T)^{-1}$ we can also give $\beta_r\fH_{r,\C_p}\beta_r^{-1}\cap\fU^-(\B_{r,\C_p})$
a structure of $B_{r,\C_p}$-module compatible with that on $\fU^-(\B_{r,\C_p})$.

By combining the previous remarks with Corollary \ref{lambdaspan}, we prove the following ``fullness'' result for the big Lie algebra $\fH_r$.

\begin{theorem}\label{betalevel}
Suppose that the representation $\rho$ is $(H_0,\Z_p)$-regular. Then there exists a nonzero ideal $\fl$ of $\I_0$, independent of $r<r_h$, such that for every such $r$ the Lie algebra $\fH_r$ contains $\fl\cdot\fsl_2(\B_r)$.
\end{theorem}

\begin{proof}
Since $U^\pm(\B_r)\cong \B_r$, we can and shall identify $\fu^+=\Q_p\cdot\Log\, G_r^\prime\cap\fU^+(\B_r)$ with a $\Q_p$-vector subspace of $\B_r$ (actually of $\I_0$), and $\fu^+_{\C_p}$ with a $\C_p$-vector subspace of $\B_{r,\C_p}$. We repeat that these spaces are independent of $r$ since
$G_r^\prime$ is, provided that $r_0\leq r<r_h$ (see Remark \ref{propsubgr}).
By Corollary \ref{lambdaspan}, $\fu^\pm\cap\I_0$ contains a basis $\{e_{i,\pm}\}_{i\in I}$ for $Q(\I_0)$ over $Q(\Lambda_{h,0})$.
The set $\{e_{i,+}\}_{i\in I}\subset\fu^+$ is a basis for $Q(\I_{0})$ over $Q(\Lambda_{h,0})$, so $\fu^+$ contains the basis of a $\Lambda_{h,0}$-lattice in $\I_0$.
By Lemma \ref{lattice} we deduce that $\Lambda_{h,0}\fu^+$ contains a nonzero ideal $\fa^+$ of $\I_0$.
Hence we also have $B_{r,\C_p}\fu^+_{\C_p}\supset B_{r,\C_p} \fa^+$. 
Now $\fa^+$ is an ideal of $\I_0$ and $B_{r,\C_p}\I_{0,\C_p}=\B_{r,\C_p}$, so $B_{r,\C_p}\fa^+=\fa^+\B_{r,\C_p}$
is an ideal in $\B_{r,\C_p}$.
We conclude that $B_{r,\C_p}\cdot\fu^+\supset\fa^+\B_{r,\C_p}$ for a nonzero ideal $\fa^+$ of $\I_0$.
We proceed in the same way for the lower unipotent subalgebra, obtaining $B_{r,\C_p}\cdot\fu^-\supset\fa^-\B_{r,\C_p}$ for some nonzero ideal $\fa^-$ of $\I_0$.

Consider now the Lie algebra $B_{r,\C_p}\fH_{\C_p}\subset \Mat_2(\B_{r,\C_p})$.
Its nilpotent subalgebras are $B_{r,\C_p}\fu^+$ and $B_{r,\C_p}\fu^-$, and we showed $B_{r,\C_p}\fu^+\supset\fa^+\B_{r,\C_p}$ and $B_{r,\C_p}\fu^-\supset\fa^-\B_{r,\C_p}$.
Denote by $\ft\subset\fsl_2$ the subalgebra of diagonal matrices over $\Z$. By taking the Lie bracket, we see that 
$[\fU^+(\fa^+\B_{r,\C_p}),\fU^-(\fa^-\B_{r,\C_p})]$ spans $\fa^+\cdot\fa^-\cdot\ft(\B_{r,\C_p})$ over $B_{r,\C_p}$.
We deduce that $B_{r,\C_p}\fH_{\C_p}\supset\fa^+\cdot\fa^-\cdot\fsl_2(\B_{r,\C_p})$.
Let $\fa=\fa^+\cdot\fa^-$. Now $\fa\cdot\fsl_2(\B_{r,\C_p})$ is a $\B_{r,\C_p}$-Lie subalgebra of $\fsl_2(\B_{r,\C_p})$.
Recall that $\beta_r\in\GL_2(\B_{r,\C_p})$; hence by stability by conjugation we have
$\beta_r\left(\fa\cdot\fsl_2(\B_{r,\C_p})\right)\beta_r^{-1}=\fa\cdot\fsl_2(\B_{r,\C_p})$.
Thus, we constructed $\fa$ such that $B_{r,\C_p}\left(\beta_r\fH_{r,\C_p}\beta_r^{-1}\right)\supset\fa\cdot\fsl_2(\B_{r,\C_p})$. 
In particular, if $\fu_{\C_p}^{\pm,\beta_r}$ denote the unipotent subalgebras of $\beta_r\fH_{r,\C_p}\beta_r^{-1}$, we have $B_{r,\C_p}\fu_{\C_p}^{\pm,\beta_r}\supset\fa\B_{r,\C_p}$ for both signs. By the discussion preceding the proposition the subalgebras $\fu_{\C_p}^{\pm,\beta_r}$ have a structure of $B_{r,\C_p}$-modules, which means that $\fu_{\C_p}^{\pm,\beta_r}=B_{r,\C_p}\fu_{\C_p}^{\pm,\beta_r}$.
We conclude that $\fu_{\C_p}^{\pm,\beta_r}\supset \beta_r\left(\fa\cdot\fU^\pm(\B_{r,\C_p})\right)\beta_r^{-1}$ for both signs. 
By the usual argument of taking the bracket, we have $\beta_r\fH_{r,\C_p}\beta_r^{-1}\supset\fa^2\cdot\fsl_2(\B_{r,\C_p})$.
We can untwist by the invertible matrix $\beta_r$ to conclude that, for $\fl=\fa^2$, we have
$\fH_{r,\C_p}\supset\fl\cdot\fsl_2(\B_{r,\C_p})$.

Let us get rid of the completed extension of scalars to $\C_p$.
For every ideal $\fa\subset\I_{0,r}$ not dividing $P_1$, let $\fH_{r,\fa}$ be the image of $\fH_r$ in $\Mat_2(\I_{0,r}/\fa\I_{0,r})$. 
Consider the two finite dimensional $\Q_p$-vector spaces $\fH_{r,\fa}$ and $\fl\cdot\fsl_2(\I_{0,r}/\fa\I_{0,r})$. Note that they are both subspaces of the finite dimensional $\Q_p$-vector space $\Mat_2(\I_{0,r}/\fa\I_{0,r})$. After extending scalars to $\C_p$, we have
\begin{equation}\label{cpincl} \fl\cdot\fsl_2(\I_{0,r}/\fa \I_{0,r})\otimes\C_p\subset\fH_{r,\fa}\otimes\C_p. \end{equation}

Let $\{e_i\}_{i\in I}$ be an orthonormal basis of the Banach space $\C_p$ over $\Q_p$, with $I$ some index set, such that $1\in\{e_i\}_{i\in I}$. Let $\{v_j\}_{j=1,...,n}$ be a $\Q_p$-basis of $\Mat_2(\I_{0,r}/\fa\I_{0,r})$ such that, for some $d\leq n$, $\{v_j\}_{j=1,...,d}$ is a $\Q_p$-basis of $\fH_{r,\fa}$. 

Let $v$ be an element of $\fl\cdot\fsl_2(\I_{0,r}/\fa\I_{0,r})$. 
Then $v\otimes 1\in\fl\cdot\fsl_2(\I_{0,r}/\fa \I_{0,r})\otimes\C_p$ and by (\ref{cpincl}) we have $v\otimes 1\in\fH_{r,\fa}\otimes\C_p$. As $\{v_j\otimes e_i\}_{1\le j\le d, i\in I}$, respectively 
$\{v_j\otimes e_i \}_{1\le j\le n,i\in I}$  is an orthonormal basis of  $\fH_{r,\fa}\otimes\C_p$, respectively of $\Mat_2(\I_{0,r}/\fa\I_{0,r})\otimes \C_p$ over $\Q_p$,  there exist $\lambda_{j,i}\in\Q_p, (j,i)\in \{1,2,...,d\}\times I$ converging to $0$ in the filter of complements of finite subsets of $\{1,2,...,d\}\times I$ such that $v\otimes 1=\sum_{j=1,...,d;\, i\in I}\lambda_{j,i}(v_j\otimes e_i)$. 

But $v\otimes 1\in \Mat_2(\I_{0,r}/\fa\I_{0,r})\otimes 1\subset \Mat_2(\I_{0,r}/\fa\I_{0,r})\otimes \C_p$ and therefore $v\otimes 1=\sum_{1\le j\le n} a_j(v_j\otimes 1)$, for some $a_j\in \Q_p$, $j=1,...,n$. By the uniqueness of a representation of an element in a 
$\Q_p$-Banach space in terms of a given orthonormal basis we have
$$
v\otimes 1=\sum_{j=1}^d a_j(v_j\otimes 1),\mbox{\quad i.e.\quad} v=\sum_{j=1}^da_jv_j\in \fH_{r,\fa}.
$$

By taking the projective limit over $\fa$, we conclude that
$$\fl\cdot \fsl_2(\B_r)\subset \fH_r.$$
\end{proof}





\begin{definition}
The Galois level of the family $\theta\colon\T_h\to\I^\circ$ is the largest ideal $\fl_\theta$ of $\I_{0}[P_1^{-1}]$ such that $\fH_{r}\supset\fl_\theta\cdot\fsl_2(\B_r)$ for all $r<r_h$.
\end{definition}

\noindent It follows by the previous remarks that $\fl_\theta$ is nonzero.

\bigskip

\section{Comparison between the Galois level and the fortuitous congruence ideal}\label{compar}

Let $\theta\colon\T_h\to\I^\circ$ be a family. We keep all the notations from the previous sections.
In particular $\rho\colon G_\Q\to\GL_2(\I^\circ)$ is the Galois representation associated with $\theta$. 
We suppose that the restriction of $\rho$ to $H_0$ takes values in $\GL_2(\I^\circ_0)$.
Recall that $\I=\I^\circ[p^{-1}]$ and $\I_0=\I_0^\circ[p^{-1}]$.
Also recall that $P_1$ is the prime of $\Lambda_{h,0}$ generated by $u^{-1}(1+T)-1$.
Let $\fc\subset\I$ be the congruence ideal associated with $\theta$.
Set $\fc_0=\fc\cap\I_0$ and $\fc_1=\fc_0\I_0[P_1^{-1}]$.
Let $\fl=\fl_\theta\subset\I_0[P_1^{-1}]$ be the Galois level associated with $\theta$. 
For an ideal $\fa$ of $\I_0[P_1^{-1}]$ we denote by $V(\fa)$ the set of prime ideals of $\I_0[P_1^{-1}]$ containing $\fa$.
We prove the following.

\begin{theorem}\label{comparison}
Suppose that
\begin{enumerate}
\item $\rho$ is $(H_0,\Z_p)$-regular;
\item there exists no pair $(F,\psi)$, where $F$ is a real quadratic field and $\psi\colon\Gal(\overline{F}/F)\to\F^\times$ 
is a character, such that $\overline{\rho}\colon G_\Q\to\GL_2(\F)\cong\Ind_F^\Q\psi$. 
\end{enumerate}
Then we have $V(\fl)=V(\fc_1)$.
\end{theorem}

Before giving the proof we make some remarks. 
Let $P$ be a prime of $\I_0[P_1^{-1}]$ and $Q$ be a prime factor of $P\I[P_1^{-1}]$. We consider $\rho$ as a representation $G_\Q\to\GL_2(\I[P_1^{-1}])$ by composing it with the inclusion $\GL_2(\I)\into\GL_2(\I[P_1^{-1}])$.
We have a representation $\rho_{Q}\colon G_\Q\to\GL_2(\I[P_1^{-1}]/Q)$ obtained by reducing $\rho$ modulo $Q$. Its restriction $\rho_{Q}\vert_{H_0}$ takes values in $\GL_2(\I_0[P_1^{-1}]/(Q\cap\I_0[P_1^{-1}]))=\GL_2(\I_0[P_1^{-1}]/P)$ and coincides with the reduction $\rho_P$ of $\rho\vert_{H_0}\colon H_0\to\GL_2(\I_0[P_1^{-1}])$ modulo $P$.
In particular $\rho_{Q}\vert_{H_0}$ is independent of the chosen prime factor $Q$ of $P\I[P_1^{-1}]$.

We say that a subgroup of $\GL_2(A)$ for some algebra $A$ finite over a $p$-adic field $K$ is \textit{small} if it admits a finite index abelian subgroup.
Let $P$, $Q$ be as above, $G_P$ be the image of $\rho_P\colon H_0\to\GL_2(\I_0[P_1^{-1}]/P)$ and $G_{Q}$ be the image of $\rho_{Q}\colon G_\Q\to\GL_2(\I[P_1^{-1}]/Q)$. 
By our previous remark $\rho_P$ coincides with the restriction $\rho_{Q}\vert_{H_0}$, so $G_P$ is a finite index subgroup of $G_{Q}$ for every $Q$.
In particular $G_P$ is small if and only if $G_{Q}$ is small for all prime factors $Q$ of $P\I[P_1^{-1}]$.

Now if $Q$ is a CM point the representation $\rho_{Q}$ is induced by a character of $\Gal(F/\Q)$ for an imaginary quadratic field $F$.
Hence $G_{Q}$ admits an abelian subgroup of index $2$ and $G_P$ is also small.

Conversely, if $G_P$ is small, $G_{Q^\prime}$ is small for every prime $Q^\prime$ above $P$.
Choose any such prime $Q^\prime$; by the argument in \cite[Prop. 4.4]{ribet2} $G_{Q^\prime}$ has an abelian subgroup of index $2$.
It follows that $\rho_{Q^\prime}$ is induced by a character of $\Gal(\overline{F}_{Q^\prime}/F_{Q^\prime})$ for a quadratic field $F_{Q^\prime}$.
If $F_{Q^\prime}$ is imaginary then ${Q^\prime}$ is a CM point.
In particular, if we suppose that the residual representation $\bar{\rho}\colon G_\Q\to\GL_2(\F)$ is not induced by a character of $\Gal(\overline{F}/F)$ for a real quadratic field $F/\Q$, then $F_{Q^\prime}$ is imaginary and ${Q^\prime}$ is CM.
The above argument proves that $G_P$ is small if and only if all points ${Q^\prime}\subset\I[P_1^{-1}]$ above $P$ are CM.

\begin{proof}
%
We prove first that $V(\fc_1)\subset V(\fl)$. Fix a radius $r<r_h$.
By contradiction, suppose that a prime $P$ of $\I_0[P_1^{-1}]$ contains $\fc_0$ but $P$ does not contain $\fl$.
Then there exists a prime factor $Q$ of $P\I[P_1^{-1}]$ such that $\fc\subset Q$.
By definition of $\fc$ we have that $Q$ is a CM point in the sense of Section \ref{congrid}, hence the representation $\rho_{\I[P_1^{-1}],Q}$ has small image in $\GL_2(\I[P_1^{-1}]/Q)$. 
Then its restriction $\rho_{\I[P_1^{-1}],Q}\vert_{H_0}=\rho_P$ also has small image in $\GL_2(\I_0[P_1^{-1}]/P)$. 
We deduce that there is no nonzero ideal $\frakI_P$ of $\I_0[P_1^{-1}]/P$ such that the Lie algebra $\fH_{r,P}$ contains $\frakI_P\cdot\fsl_2(\I_0[P_1^{-1}]/P)$.

Now by definition of $\fl$ we have $\fl\cdot\fsl_2(\B_r)\subset\fH_r$.
Since reduction modulo $P$ gives a surjection $\fH_r\to\fH_{r,P}$, by looking at the previous inclusion modulo $P$ we find $\fl\cdot \fsl_2(\I_{0,r}[P_1^{-1}]/P\I_{0,r}[P_1^{-1}])\subset\fH_{r,P}$.
If $\fl\not\subset P$ we have $\fl/P\neq 0$, which contradicts our earlier statement.
We deduce that $\fl\subset P$.

We prove now that $V(\fl)\subset V(\fc_1)$.
Let $P\subset\I_0[P_1^{-1}]$ be a prime containing $\fl$. Recall that $\I_0[P_1^{-1}]$ has Krull dimension one, so $\kappa_{P}=\I_0[P_1^{-1}]/P$ is a field.
Let $Q$ be a prime of $\I[P_1^{-1}]$ above $P$. As before $\rho$ reduces to representations $\rho_Q\colon G_\Q\to\GL_2(\I[P_1^{-1}]/Q)$ and $\rho_P\colon H_0\to\GL_2(\I_0[P_1^{-1}]/P)$.
Let $\fP\subset\I_0[P_1^{-1}]$ be the $P$-primary component of $\fl$ and let $\fA$ be an ideal of $\I_0[P_1^{-1}]$ containing $\fP$ such that
the localization at $P$ of $\fA/\fP$ is one-dimensional over $\kappa_P$. 
Choose any $r<r_h$.
Let $\fs=\fA/\fP\cdot\fsl_2(\I_{0,r}[P_1^{-1}]/\fP)\cap\fH_{r,\fP}\subset
\fA/\fP\cdot\fsl_2(\I_{0,r}[P_1^{-1}]/\fP)$.

We show that $\fs$ is stable under the adjoint action $\Ad(\rho_Q)$ of $G_\Q$. Let $\fQ$ be the $Q$-primary component of $\fl\cdot\I[P_1^{-1}]$.
Recall that $\fH_{r,\fP}$ is the Lie algebra associated with the pro-$p$ group $\im\,\rho_{r,\fQ}\vert_{H_0}\cap\Gamma_{\GL_2(\I_{0,r}[P_1^{-1}]/\fP)}(p)\subset\GL_2(\I_{0,r}[P_1^{-1}]/\fP)$. Since this group is open in $\im\,\rho_{r,\fQ}\subset\GL_2(\I_r[P_1^{-1}]/\fQ)$, the Lie algebra associated with $\im\,\rho_{r,\fQ}$ is again $\fH_{r,\fP}$.
In particular $\fH_{r,\fP}$ is stable under $\Ad(\rho_Q)$. 
Since $\fH_{r,\fP}\subset\fsl_2(\I_{0,r}[P_1^{-1}]/\fP)$ we have $\fA/\fP\cdot\fsl_2(\I_{0,r}[P_1^{-1}]/\fP)\cap\fH_{r,\fP}=\fA/\fP\cdot\fsl_2(\I_r[P_1^{-1}]/\fQ)\cap\fH_{r,\fP}$. Now $\fA/\fP\cdot\fsl_2(\I_r[P_1^{-1}]/\fQ)$ is clearly stable under $\Ad(\rho_Q)$, so the same is true for $\fA/\fP\cdot\fsl_2(\I_r[P_1^{-1}]/\fQ)\cap\fH_{r,\fP}$, as desired. 

We consider from now on $\fs$ as a Galois representation via $\Ad(\rho_Q)$. 
By the proof of Theorem \ref{betalevel} we can assume, possibly considering a sub-Galois 
representation, that $\fH_r$ is a $\B_r$-submodule of $\fsl_2(\B_r)$ 
containing $\fl\cdot\fsl_2(\B_r)$ but not $\fa\cdot\fsl_2(\B_r)$ for any $\fa$ strictly bigger than $\fl$. 
This allows us to speak of the localization $\fs_P$ of $\fs$ at $P$.
Note that, since $\fP$ is the $P$-primary component of $\fl$ and $\fA_P/\fP_P\cong\kappa_P$, when $P$-localizing we find $\fH_{r,P}\supset\fP_{P}\cdot\fsl_2(\B_{r,P})$ and $\fH_{r,P}\not\supset\fA_P\cdot\fsl_2(\B_{r,P})$.

The localization at $P$ of $\fa/\fP\cdot\fsl_2(\I_{0,r}[P_1^{-1}]/\fP)$ is $\fsl_2(\kappa_P)$, so $\fs_P$ is contained in $\fsl_2(\kappa_P)$. It is a $\kappa_P$-representation $G_\Q$ (via $\Ad(\rho_Q)$) of dimension at most $3$. We distinguish various cases following its dimension.

%
%
We cannot have $\fs_P=0$. By exchanging the quotient with the localization we would obtain $(\fA_{P}\cdot\fsl_2(\B_{r,P})\cap\fH_{r,P})/\fP_{P}=0$. By Nakayama's lemma $\fA_{P}\cdot\fsl_2(\B_{r,P})\cap\fH_{r,P}=0$, which is absurd since $\fA_{P}\cdot\fsl_2(\B_{r,P})\cap\fH_{r,P}\supset\fP_{P}\cdot\fsl_2(\B_{r,P})\neq 0$.

We also exclude the $3$-dimensional case. If $\fs_{P}=\fsl_2(\kappa_{P})$, by exchanging the quotient with the localization we obtain  $(\fA_{P}\cdot\fsl_2(\B_{r,P})\cap\fH_{r,P})/\fP_{P}=(\fA_P\cdot\fsl_2(\I_{0,r,P}[P_1^{-1}]))/\fP_{P}\I_{0,r,P}[P_1^{-1}]$, because
$\fA_{P}\I_{0,r,P}[P_1^{-1}]/\fP_P\I_{0,r,P}[P_1^{-1}]=\left(\I_{0,r,P}[P_1^{-1}]/\fP_{P}\I_{0,r,P}[P_1^{-1}]\right)$ and this is isomorphic to $\kappa_{P}$.
By Nakayama's lemma we would conclude that $\fH_{r,P}\supset\fA\cdot\fsl_2(\B_{r,P})$, which is absurd.

We are left with the one and two-dimensional cases. If $\fs_P$ is two-dimensional we can always replace it by its orthogonal 
in $\fsl_2(\kappa_P)$ which is one-dimensional; indeed the action of $G_\Q$ via $\Ad(\rho_Q)$ is isometric with respect to the scalar product $\Tr(XY)$ on $\fsl_2(\kappa_P)$.

Suppose that $\fsl_2(\kappa_P)$ contains a one-dimensional stable subspace.
Let $\phi$ be a generator of this subspace over $\kappa_P$.
Let $\chi\colon G_\Q\to\kappa_P$ denote the character satisfying $\rho_Q(g)\phi\rho_Q(g)^{-1}=\chi(g)\phi$ for all $g\in G_\Q$.
Now $\phi$ induces a nontrivial morphism of representations $\rho_Q\to\rho_Q\otimes\chi$. Since $\rho_P$ and $\rho_Q\otimes\chi$ are irreducible, by Schur's lemma $\phi$ must be invertible. Hence we obtain an isomorphism $\rho_Q\cong\rho_Q\otimes\chi$.
By taking determinants we see that $\chi$ must be quadratic. If $F_0/\Q$ is the quadratic extension fixed by $\ker\chi$, then $\rho_Q$ is induced by a character $\psi$ of $\Gal(\overline{F_0}/F_0)$. 
By assumption the residual representation $\rho_{\fm_\I}\colon G_\Q\to\GL_2(\F)$ is not of the form $\Ind_F^\Q\psi$ for a real quadratic field $F$ and a character $\Gal(\overline{F}/F)\to\F^\times$. We deduce that $F_0$ must be imaginary, so $Q$ is a CM point by Remark \ref{cmlocus}(1). By construction of the congruence ideal $\fc\subset Q$ and $\fc_0\subset Q\cap\I_0[P_1^{-1}]=P$.
\end{proof}

We prove a corollary.

\begin{corollary}
If the residual representation $\overline{\rho}\colon G_\Q\to\GL_2(\F)$ is not dihedral then $\fl=1$.
\end{corollary}

\begin{proof}
Since $\overline{\rho}$ is not dihedral there cannot be any CM point on the family $\theta\colon\T_h\to\I^\circ$.
By Theorem \ref{comparison} we deduce that $\fl$ has no nontrivial prime factor, hence it is trivial.
\end{proof}

\begin{remark}
Theorem \ref{comparison} gives another proof of Proposition \ref{finiteCM}. Indeed the CM points of a family $\theta\colon\T_h\to\I^\circ$ correspond to the prime factors of its Galois level, which are finite in number.
\end{remark}

%

We also give a partial result about the comparison of the exponents of each prime factor in $\fc_1$ and $\fl$.
This is an analogous of what is proved in \cite[Th. 8.6]{hida} for the ordinary case; our proof also relies on the strategy there.
For every prime $P$ of $\I_0[P_1^{-1}]$ we denote by $\fc_1^P$ and $\fl^P$ the $P$-primary components respectively of $\fc_1$ and $\fl$.

\begin{theorem}\label{exponents}
Suppose that $\overline{\rho}$ is not induced by a character of $G_F$ for a real quadratic field $F/\Q$.
We have $(\fc_1^P)^2\subset\fl^P\subset\fc_1^P$.
\end{theorem}

\begin{proof}
The inclusion $\fl^P\subset\fc_1^P$ is proved in the same way as the first inclusion of Theorem \ref{comparison}.

We show that the inclusion $(\fc_1^P)^2\subset\fl^P$ holds.
If $\fc_1^P$ is trivial this reduces to Theorem \ref{comparison}, so we can suppose that $P$ is a factor of $\fc_1$. 
Let $Q$ denote any prime of $\I[P_1^{-1}]$ above $P$. Let $\fc_1^Q$ be a $Q$-primary ideal of $\I[P_1^{-1}]$ satisfying $\fc_1^Q\cap\I_0[P_1^{-1}]=\fc_1^P$.
Since $P$ divides $\fc_1$, $Q$ is a CM point, so we have an isomorphism $\rho_P\cong\Ind_F^\Q\psi$ for an imaginary quadratic field $F/\Q$ and a character $\psi\colon G_F\to\C_p^\times$.
Choose any $r<r_h$.
Consider the $\kappa_P$-vector space $\fs_{\fc_1^P}=\fH_r\cap\fc_1^P\cdot\fsl_2(\I_{0,r})/\fH_r\cap\fc_1^PP\cdot\fsl_2(\I_{0,r})$.
We see it as a subspace of $\fsl_2(\fc_1^P/\fc_1^PP)\cong\fsl_2(\kappa_{P})$. By the same argument as in the proof of Theorem \ref{comparison}, $\fs_{\fc_1^P}$ is stable under the adjoint action $\Ad(\rho_{\fc_1^QQ})\colon G_\Q\to\Aut(\fsl_2(\kappa_P))$.

Let $\chi_{F/\Q}\colon G_\Q\to\C_p^\times$ be the quadratic character defined by the extension $F/\Q$.
Let $\varepsilon\in G_\Q$ be an element projecting to the generator of $\Gal(F/\Q)$. Let $\psi^\varepsilon\colon G_F\to\C_p^\times$ be given by $\psi^\varepsilon(\tau)=\psi(\varepsilon\tau\varepsilon^{-1})$. Set $\psi^-=\psi/\psi^\varepsilon$.
Since $\rho_Q\cong\Ind_F^\Q\psi$, we have a decomposition $\Ad(\rho_Q)\cong\chi_{F/\Q}\oplus\Ind_F^\Q\psi^-$, where the two factors are irreducible.
Now we have three possibilities for the Galois isomorphism class of $\fs_{\fc_1^P}$: it is either that of $\Ad(\rho_P)$ or that of one of the two irreducible factors.

If $\fs_{\fc_1^P}\cong\Ad(\rho_Q)$, then as $\kappa_P$-vector spaces $\fs_{\fc_1^P}=\fsl_2(\kappa_P)$. 
We proceed as in the proof of Theorem \ref{comparison} to obtain $\fs_{\fc_1^P}=\fsl_2(\kappa_P)$. 
By Nakayama's lemma $\fH_r\supset\fc_1^P\cdot\fsl_2(\B_{r})$. 
This implies $\fc_1^P\subset\fl^P$, hence $\fc_1^P=\fl^P$ in this case.

If $\fs_{\fc_1^P}$ is one-dimensional then we proceed as in the proof of Theorem \ref{comparison} to show that $\rho_{\fc_1^QQ}\colon G_\Q\to\GL_2(\I_r[P_1^{-1}]/\fc_1^QP\I_r[P_1^{-1}])$ is induced by a character $\psi_{\fc_1^QQ}\colon G_F\to\C_p^\times$. In particular the image of $\rho_{\fc_1^PP}\colon H\to\GL_2(\I_{0,r}[P_1^{-1}]/\fc_1^PP\I_{0,r})$ is small.
This is a contradiction, since $\fc_1^P$ is the $P$-primary component of $\fc_1$, hence it is the smallest $P$-primary ideal $\fA$ of $\I_{0,r}[P_1^{-1}]$ such that the image of $\rho_\fA\colon G_\Q\to\GL_2(\I_r[P_1^{-1}]/\fA\I_r[P_1^{-1}])$ is small.

Finally, suppose then that $\fs_{\fc_1^P}\cong\Ind_F^\Q\psi^-$.
Let $d=\diag(d_1,d_2)\in\rho(G_\Q)$ be the image of a $\Z_p$-regular element.
Since $d_1$ and $d_2$ are nontrivial modulo the maximal ideal of $\I_0^\circ$, the image of $d$ modulo $\fc_1^QQ$ is a nontrivial diagonal element $d_{\fc_1^QQ}=\diag(d_{1,\fc_1^QQ},d_{2,\fc_1^QQ})\in\rho_{\fc_1^QQ}(G_\Q)$.
We decomposte $\fs_{\fc_1^P}$ in eigenspaces for the adjoint action of $d_{\fc_1^QQ}$: we write $\fs_{\fc_1^P}=\fs_{\fc_1^P}[a]\oplus\fs_{\fc_1^P}[1]\oplus\fs_{\fc_1^P}[a^{-1}]$, where $a=d_{1,\fc_1^QQ}/d_{2,\fc_1^QQ}$.
Now $\fs_{\fc_1^P}[1]$ is contained in the diagonal torus, on which the adjoint action of $G_\Q$ is given by the character $\chi_{F/\Q}$.
Since $\chi_{F/\Q}$ does not appear as a factor of $\fs_{\fc_1^P}$, we must have $\fs_{\fc_1^P}[1]=0$.
This implies that $\fs_{\fc_1^P}[a]\neq 0$ and $\fs_{\fc_1^P}[a^{-1}]\neq 0$.
Since $\fs_{\fc_1^P}[a]=\fs_{\fc_1^P}\cap\fu^+(\kappa_P)$ and $\fs_{\fc_1^P}[a^{-1}]=\fs_{\fc_1^P}\cap\fu^-(\kappa_P)$, we deduce that $\fs_{\fc_1^P}$ contains nontrivial upper and lower nilpotent elements $\overline{u^+}$ and $\overline{u^-}$.
Then $\overline{u^+}$ and $\overline{u^-}$ are the images of some elements $u^+$ and $u^-$ of $\fH_r\cap\fc_1^P\cdot\fsl_2(\I_{0,r}[P_1^{-1}])$ nontrivial modulo $\fc_1^PP$.
The Lie bracket $t=[u^+,u^-]$ is an element of $\fH_r\cap\ft(\I_{0,r}[P_1^{-1}])$ (where $\ft$ denotes the diagonal torus) and it is nontrivial modulo $(\fc_1^P)^2P$.
Hence the $\kappa_P$-vector space $\fs_{(\fc_1^P)^2}=\fH_r\cap(\fc_1^P)^2\cdot\fsl_2(\I_{0,r,\C_p}[P_1^{-1}])/\fH_r\cap(\fc_1^P)^2P\cdot\fsl_2(\I_{0,r,\C_p}[P_1^{-1}])$ contains nontrivial diagonal, upper nilpotent and lower nilpotent elements, so it is three-dimensional.
By Nakayama's lemma we conclude that $\fH_r\supset(\fc_1^P)^2\cdot\fsl_2(\I_{0,r}[P_1^{-1}])$, so $(\fc_1^P)^2\subset\fl^P$.
\end{proof}


\bigskip
\bigskip

\begin{small}
\textsc{Universit\'e Paris 13,
Sorbonne Paris Cit\'e, LAGA, CNRS (UMR 7539),
99, avenue J.-B. Cl\'ement,
F-93430, Villetaneuse, France.}

\textit{E-mail addresses}: \url{conti@math.univ-paris13.fr}, \url{tilouine@math.univ-paris13.fr}

\bigskip

\textsc{Department of Mathematics and Statistics, Concordia University, Montreal, Canada and
Dipartimento di Matematica, Universita degli Studi di Padova, Padova, Italy.}

\textit{E-mail address}: \url{adrian.iovita@concordia.ca}
\end{small}

\end{document}